\documentclass[11pt]{article}
\usepackage{amsmath,amsthm,amsfonts,amssymb,mathrsfs,bm,graphicx,stmaryrd,hyperref}
\usepackage{mathtools}
\usepackage[usenames]{color}
\usepackage[letterpaper,hmargin=1.0in,vmargin=1.0in]{geometry}
\parskip 	\smallskipamount
\usepackage{tocloft}
\cftsetindents{section}{-1em}{2em}
\newtheorem{theorem}{Theorem}[section]
\newtheorem{lemma}[theorem]{Lemma}
\newtheorem{corollary}[theorem]{Corollary}
\newtheorem{proposition}[theorem]{Proposition}

\newtheorem{remark}[theorem]{Remark}

\newtheorem{definition}[theorem]{Definition}

\newtheorem{maintheorem}{Theorem}

\def\ind{{\mathbf 1}}
\def\N{\mathbb{N}}
\def\L{\mathbb{L}}
\def\P{\mathbb{P}}
\def\Z{\mathbb{Z}}
\def\R{\mathbb{R}}

\def\E{\mathbb{E}}
\def\l{\ell}
\DeclarePairedDelimiter\floor{\lfloor}{\rfloor}
\newcommand{\cF}{\mathcal{F}}

\newcommand{\cf}{\mathcal{F}}
\newcommand{\cA}{\mathcal{A}}
\newcommand{\cB}{\mathcal{B}}

\newcommand{\ce}{\mathcal{E}}
\begin{document}
\title{Invariant Measures for TASEP with a Slow Bond}

\author{Riddhipratim Basu
\thanks{Department of Mathematics, Stanford University. Email: rbasu@stanford.edu}
 \and
Sourav Sarkar
\thanks{Department of Statistics, University of California, Berkeley. Supported by Lo\'{e}ve Fellowship. Email: souravs@berkeley.edu}
\and
Allan Sly
\thanks{Department of Mathematics, Princeton University and Department of Statistics, University of California, Berkeley. Email: asly@math.princeton.edu}
}

\date{}
\maketitle
           
\begin{abstract}
Totally Asymmetric Simple Exclusion Process (TASEP) on $\Z$ is one of the classical exactly solvable models in the KPZ universality class. We study the ``slow bond" model, where TASEP on $\Z$ is imputed with a slow bond at the origin. The slow bond increases the particle density immediately to its left and decreases the particle density immediately to its right. Whether or not this effect is detectable in the macroscopic current started from the step initial condition has attracted much interest over the years and this question was settled recently in \cite{BSS14} where it was shown that the current is reduced even for arbitrarily small strength of the defect. Following non-rigorous physics arguments in \cite{JL1,JL2} and some unpublished works by Bramson, a conjectural description of properties of invariant measures of TASEP with a slow bond at the origin was provided by Liggett in \cite{Lig99}. We establish Liggett's conjectures and in particular show that TASEP with a slow bond at the origin, starting from step initial condition, converges in law to an invariant measure that is asymptotically close to product measures with different densities far away from the origin towards left and right. Our proof exploits the correspondence between TASEP and the last passage percolation on $\Z^2$ with exponential weights and uses the understanding of geometry of maximal paths in those models. 
\end{abstract} 

\section{Introduction}
Totally Asymmetric Simple Exclusion Process (TASEP) is a classical interacting particle system in statistical mechanics. On the line, the dynamics is as follows: each particle jumps to the right at rate one provided the site to its right is empty. This process has been studied in detail for more than past forty years on both statistical physics and probability literature, and a rich understanding of its behaviour has emerged. Stationary measures for TASEP was identified by Liggett \cite{Lig76} as early as in 1976 when he showed that product Bernoulli measures are all non-trivial extremal stationary measures for the TASEP dynamics. This and a sequence of works \cite{Lig73, Lig75, Lig77,BLM02} has characterised the stationary measures as well has proved ergodic theorems for symmetric and asymmetric exclusion processes for various different settings. Utilizing this progress, Rost \cite{Ro81} in 1981 evaluated the asymptotic current and hydrodynamic density profile when the process starts from the step initial condition, i.e., with one particle each at every nonpositive site of $\Z$ and no particles at positive sites. More recently, TASEP was identified to be \cite{Jo99} one of the canonical \emph{exactly solvable} models in the so-called KPZ universality class, and thus very fine information about the process was obtained using exact determinantal formulae that included the Tracy-Widom scaling limits for the current fluctuations.

It has been a topic of contemporary interest in equilibrium and non-equilibrium statistical mechanics to understand how the macroscopic behaviour of a system changes if some local, microscopic defect of arbitrarily small strength is introduced. A specific such model was introduced in the context of TASEP by Janowsky and Lebowitz \cite{JL1,JL2} who considered TASEP with a slow bond at the origin where a particle jumping from the origin jumps at some rate $r<1$. It is easy to see that for $r$ small this model has smaller asymptotic current started from the step initial condition, and Janowsy and Lebowitz asked whether the same happens for an arbitrarily small strength of defect, i.e., values of $r$ arbitrarily close to one. Over two decades, there were disagreements among physicists about what the answer should be with different groups predicting different answers, and this problem came to be known as the ``slow bond problem". As is typical for exactly solvable models, much of the detailed analysis of TASEP is non-robust, i.e., the analysis breaks down under minor modifications to the model. In particular, the study of the model with a slow bond is no longer facilitated by the exact formulae, and even the stationary distributions are non-explicit, making the study of the model much harder. This question was settled very recently in \cite{BSS14} where it was shown using a geometric approach together with the exactly solvable ingredients from the TASEP (without the slow bond) that the local current is restricted for any arbitrary small blockage parameter. That is, for any value of $r<1$, it was established that the limiting current is strictly less than $\frac{1}{4}$, which is the corresponding value for regular TASEP.

In this paper, we develop further the geometric techniques introduced in \cite{BSS14} to study the stationary measures for TASEP with a slow bond. Following the works \cite{JL1,JL2} and some unpublished works by Bramson, the conjectural picture that emerged (modulo the affirmative answer to the slow bond problem which has now been established) is described in Liggett's 1999 book \cite[p. 307]{Lig99}. The distribution of regular TASEP started with the step initial condition converges to the invariant product Bernoulli measure with density $\frac{1}{2}$. The slowdown due to the slow bond implies that there is a long range effect near the origin where the region to the right of origin is sparser and there is a traffic jam to the left of the slow bond with particle density higher than a half. However, it was conjectured that as one moves far away from the origin, the distribution becomes close to a product measure albeit with a different density $\rho <\frac{1}{2}$ to the right of the origin and $\rho' >\frac{1}{2}$ to the left of the origin. Our contribution in this paper is to establish this picture rigorously and thus answering Liggett's question described above; see Theorem \ref{t:mainst}, Theorem \ref{t:otherrho} and Corollary \ref{c:partc} below.

As in \cite{BSS14} our argument is also based on the connection between TASEP and directed last passage percolation (DLPP) on $\Z^2$ with Exponential passage times, which will be recalled below in Subsection \ref{s:lpp}. We study the geometry of the geodesics (maximal paths) in the last passage percolation models corresponding to both TASEP and TASEP with a slow bond. We use the result from \cite{BSS14} to establish quantitative estimates about pinning of certain point-to-point geodesics in the slow bond model. This establishes certain correlation decay and mixing properties for the average occupation measures which implies the existence of a limiting invariant measure. The heart of the argument showing that this invariant measure is close to product measure far away from the origin is another analysis of the geometry of the geodesics in the Exponential directed last passage percolation model, together with a coupling between the slow bond process and a stationary TASEP. We use crucially a result about coalescence of geodesic in Exponential LPP, obtained in the companion paper \cite{BSS17++}.

We now move towards formal definitions and precise statement of our results.

\subsection{Formal Definitions and Main Result}
Formally TASEP is defined as a continuous time Markov process with the state space $\{0,1\}^{\Z}$. Let $\{\eta_{t}\}_{t\geq 0}$ denote the particle configuration at time $t$, i.e., for $t\geq 0$ and $x\in \Z$, let $\eta_t(x)=1$ or $0$ depending on whether there is a particle at time $t$ on site $x$ or not. Let $\delta_{x}$ denote the particle configuration with a single particle at site $x$. For a particle configuration $\eta=(\eta(x):x\in \Z)$, denote by $\eta^{x,x+1}$ the particle configuration $\eta-\delta_{x}+\delta_{x+1}$, i.e., where a particle has jumped from the site $x$ to the site $x+1$. TASEP dynamics defines a Markov process with the generator given by 
$$\mathscr{L} f(\eta)= \sum_{x\in \Z} \eta(x)(1-\eta(x+1))(f(\eta^{x,x+1})-f(\eta)).$$  

Let $\{P_{t}\}_{t\geq 0}$ denote the corresponding semigroup. A probability measure $\nu$ on $\{0,1\}^{\Z}$ is called an \emph{invariant measure} or \emph{stationary measure} for TASEP if $\E_{\nu}(f)=\E_{\nu} (P_{t}f)$ for all $t\geq 0$ and for all bounded continuous functions $f$ on the state space. Denoting the distribution of $\eta_{t}$ when $\eta_{0}$ is distributed according to $\nu$ by $\nu P_t$ the above says that for an invariant measure $\nu$ one has $\nu P_t= \nu$. The invariant measures for TASEP can be characterised, see Section \ref{s:bg} below.

\subsubsection{TASEP with a Slow Bond at the Origin}
We shall consider the TASEP dynamics in presence of the following microscopic defect: for a fixed $r<1$ consider the exclusion dynamics where every particle jumping out of the origin jumps at a slower rate $r<1$. Formally this is a Markov process with the generator 
$$\mathscr{L}^{(r)} f(\eta)=r\eta(0)(1-\eta(1))(f(\eta^{0,1})-f(\eta)) +\sum_{x\neq 0} \eta(x)(1-\eta(x+1))(f(\eta^{x,x+1})-f(\eta)).$$
Denote the corresponding semigroup by $\{P^{(r)}_{t}\}_{t\geq 0}$. For the rest of this paper we shall treat $r$ as a fixed quantity arbitrarily close to one. It turns out even a microscopic defect of arbitrarily small strength has a macroscopic effect to the system, see Section \ref{s:sb} for more details of this model. In particular, their is long range correlation near the origin and the invariant measures for this model does not admit any explicit description unlike regular TASEP. As mentioned above, following \cite{JL1,JL2} and unpublished works by Bramson, Liggett \cite[p.307]{Lig99} described the conjectural behaviour for the invariant measures for TASEP with a slow bond at the origin. Our main result in this paper is to confirm this conjecture. We now introduce notations and definitions necessary for stating our results.

The initial condition $\eta_0= \mathbf{1}_{(-\infty,0]}$ (i.e., one particle each at all sites $x\leq 0$ and no particles at sites $x>0$) is particularly important to study of TASEP and is called \emph{step initial condition}. For $0<\rho <1$, let $\nu_{\rho}$ denote the product Bernoulli measure on the space of particle configurations $\{0,1\}^\Z$ with density $\rho$; i.e., $\nu_{\rho}(\eta(x)=1)=\rho$ independently for all $x\in \Z$. For measures $\nu$ on $\{0,1\}^{\Z}$, $E\subseteq \Z$ and $A\subseteq \{0,1\}^{E}$ we shall set without loss of generality $\nu(A)=\nu(A\times \{0,1\}^{\Z\setminus E})$. Also for $x\in Z$, let $A^x$ denote the subset of $\{0,1\}^{x+E}$ obtained by a co-ordinate wise translation by $x$.
We shall need the following definition. 

\begin{definition}
A probability measure $\nu$ on the configuration space  $\{0,1\}^{\Z}$ is said to be asymptotically equivalent to $\nu_{\rho}$ at $\infty$ (resp.\ at $-\infty$) if the following holds: for every finite $E$ and a subset $A$ of $\{0,1\}^{E}$ we have 
$$\nu (A^{k}) \to \nu_{\rho}(A)$$
as $k\to \infty$ (resp.\ $k\to -\infty$).  
\end{definition}

We are now ready to state our main result which solves part (a) of the sequence of questions about the invariant measures for TASEP with a slow bond in \cite[p.307]{Lig99}. The other parts of the conjecture also follow from this work and have been outlined in this paper, see Theorem \ref{t:otherrho} and Corollary \ref{c:partc} for the statements of parts (b) and (c) of the question respectively.

\begin{maintheorem}
\label{t:mainst}
For every $r<1$, there exists a measure $\nu_{*}$ on $\{0,1\}^{\Z}$ and $\rho< \frac{1}{2}$ both depending on $r$ such that $\nu_{*}$ is an invariant measure for the  Markov process with generator $\mathscr{L}^{(r)}$ and $\nu_{*}$ is asymptotically equivalent to $\nu_{\rho}$ (resp.\ $\nu_{1-\rho}$) at $\infty$ (resp.\ $-\infty$). Furthermore, started from the step initial condition $\eta_0$, the process converges weakly to $\nu_{*}$, i.e., 
$\delta_{\eta_0} P_t^{(r)}\Rightarrow \nu_{*}$ as $t\to \infty$. 
\end{maintheorem}

For the rest of this paper we shall keep $r<1$ fixed. The density $\rho$ in the statement of the theorem is not an explicit function of $r$; however, we can evaluate $\rho$ explicitly in terms of the asymptotic current in the process with a slow bond; see Remark \ref{r:rho} below.

\subsection{Background}
\label{s:bg}
As mentioned before, studying the invariant measures is crucial for understanding many particle systems such as TASEP especially at infinite volume. It is a well known fact that the set of all invariant measures is a compact convex subset of the set of all probability measures in the topology of weak convergence and hence it suffices to study the extremal invariant measures. In \cite{Lig76}, Liggett identified the set of all invariant measures for TASEP, and showed that apart from a few trivial measures, the extremal invariant measures for TASEP are the Bernoulli product measures $\{\nu_{\rho}: \rho\in (0,1)\}$. Notice that these are all translation invariant stationary measures for TASEP, and the measure $\nu_{\rho}$ corresponds to the stationary current $J=\rho(1-\rho)$ where $J$ denotes the rate at which particles cross a bond. Hence the maximum possible value of the stationary current is $\frac{1}{4}$. In particular, started with the step initial condition $\eta_0= \mathbf{1}_{(-\infty,0]}$ TASEP converges in distribution to the stationary measure $\nu_{\frac{1}{2}}$ (cf.\ Theorem 3.29 in part III of \cite{Lig99}). Rost \cite{Ro81} established the hydrodynamic density profile and asymptotic current for TASEP started with step initial condition. Let $L_{n}$ denote time it takes for $n$ particles to cross the origin. Rost \cite{Ro81} established
\begin{equation}
\label{e:lln}
\lim_{n\to \infty}\frac{\E L_{n}}{n}=4. 
\end{equation}
Observe that the reciprocal of the number on the right hand side above denotes the asymptotic rate at which particles cross the origin, which in this case matches the stationary current of the limiting measure $\nu_{\frac{1}{2}}$. 

\subsubsection{The Slow Bond Problem}
\label{s:sb}
TASEP with a slow bond at the origin was introduced by Janowsky and Lebowitz \cite{JL1,JL2} in an attempt to understand non-equilibrium steady states. Recall that $r<1$ is the rate at which particles jump at the origin in the slow bond model. For values of $r$ close to $1$ this model can be viewed as introducing a small microscopic defect in TASEP. The decrease in jump rate at the origin will increase particle density to the immediate left of the slow bond whereas decrease particle density to its immediate right. It is not difficult to see that in addition to this local effect, for values of $r$ sufficiently small, the introduction of the slow bond has a macroscopic effect and reduces the value of the asymptotic current. What is not clear, however, is whether for arbitrarily small strength of this defect (i.e., for values of $r$ arbitrarily close to $1$) the local effect of the slow bond is destroyed by the fluctuations in the bulk thus making the slow bond macroscopically undetectable. Specifically, one asks the following. Let $L_{n}^{(r)}$ denote the time it takes for the $n$ particles to cross the origin; is $\lim_{n\to \infty} \frac{L_n^{(r)}}{n}$ strictly larger that $4$ for all values of $r<1$, or there is a critical value $r_c<1$ below which this is observed? This question came to be known as the \emph{slow bond problem} in the statistical physics literature. 

For more than two decades this question had proved controversial with various groups of physicists arriving at competing conclusions based on empirical simulation studies and heuristic argument. In \cite{JL2}, a mean field approximation argument was given suggesting that indeed the critical value $r_c=1$, whereas others e.g.\ \cite{htm03,mmmthn} argued the opposite. Despite several progress in rigorous analysis of the model \cite{crez,timo} (see also \cite{clst13} for a fuller description of the history of the problem) this question remained unanswered until very recently. One difficulty in analyzing such systems comes from the fact that effect of any local perturbation is felt at all scales because the system carries conserved quantities. Very recently, this question was settled in \cite{BSS14} using a geometric approach together with estimates coming from the exactly solvable nature of TASEP  (see Section \ref{s:int} below). 

\begin{theorem}[\cite{BSS14}, Theorem 2]
\label{t:lln}
For any $r<1$, there exists $\varepsilon=\varepsilon(r)>0$ such that 
$$\lim_{n\to \infty} \frac{L_n^{(r)}}{n}=4+\varepsilon.$$
\end{theorem}

This result does not readily yield any information about the invariant measures for TASEP with a slow bond. Because of the long range effect of the slow bond, it is expected that any invariant measure for this process must have complicated correlation structure near the origin. Theorem \ref{t:mainst} establishes that starting from the step initial condition the process converges to such an invariant measure $\nu_{*}$ which, as conjectured, is asymptotically equivalent to product measures $\nu_{\rho}$ at $\infty$ and $\nu_{1-\rho}$ at $-\infty$ for some $\rho=\rho(r)$ strictly smaller than $\frac{1}{2}$.
Although we do not have an explicit formula for $\rho$ in terms of $r$, it can be related to $\varepsilon$ in a simple manner. 

\begin{remark}
\label{r:rho}
Let $\varepsilon=\varepsilon(r)$ be as in Theorem \ref{t:lln}. Then $\rho$ in Theorem \ref{t:mainst} is given by the unique real number less than $\frac{1}{2}$ satisfying $$\rho(1-\rho)=\frac{1}{4+\varepsilon}.$$
\end{remark}
It is easy to see why the above should be true. Since the current at any site under the stationary measure $\nu_{*}$ is the same; the current at the origin (equal to $\frac{1}{4+\varepsilon}$ by Theorem \ref{t:lln}) should be same as the current at some far away site both to the left and right. Recall that the stationary current under $\nu_{\rho}$ is $\rho(1-\rho)$ and thus Theorem \ref{t:mainst} suggests that $\rho(1-\rho)=\frac{1}{4+\varepsilon}$. We shall see from our proof of Theorem \ref{t:mainst} that this is indeed the case. Similar considerations give the following easy corollary of Theorem \ref{t:lln} answering part (c) of Liggett's question in  \cite[p.307]{Lig99}.

\begin{corollary}
\label{c:partc}
Let $\rho=\rho(r)$ be as above. For $p>\rho$ there does not exist an invariant measure of the Markov process with generator $\mathscr{L}^{(r)}$ that is equivalent to $\nu_{p}$ at $\infty$ (resp.\ $\nu_{1-p}$ at $-\infty$). 
\end{corollary}


Arguments similar to our proof of Theorem \ref{t:mainst} can also be used to settle part (b) of Liggett's question in  \cite[p.307]{Lig99} where Liggett asks if for any $p<\rho(r)$ (resp.\ for any $p> 1-\rho(r)$) there exists an invariant measure of $\mathscr{L}^{(r)}$ that is equivalent to $\nu_{p}$ at $\pm \infty$ (resp.\ $\nu_{1-p}$ at $\pm \infty$). We show that such an invariant measure is obtained in the limit if the process is started from  product $\mbox{Ber}(p)$ stationary initial condition. 

\begin{theorem}
\label{t:otherrho}
Let $p<\rho (r)$ or $p> 1-\rho(r)$ be fixed. Then there exists an invariant measure $\nu^*_{p}$ of the Markov process with generator $\mathscr{L}^{(r)}$ such that $\nu_{p}P_{t}^{(r)}\Rightarrow \nu^*_{p}$ and $\nu^*_p$ is asymptotically equivalent to $\nu_{p}$ at $\pm \infty$.
\end{theorem}

Theorem \ref{t:mainst}, Theorem \ref{t:otherrho} and Corollary \ref{c:partc} answer all parts of Liggett's question. However we do not identify all invariant measures for TASEP with a slow bond. A natural question asked by Liggett \cite{Lig16} is whether or not the invariant measures given in Theorem \ref{t:mainst} and Theorem \ref{t:otherrho} are the only nontrivial extremal invariant measures of the process. Another question of interest again pointed out by Liggett \cite{Lig16} is whether similar results hold for other translation invariant exclusion systems with positive drift, i.e., ASEP. Our techniques do not apply as there is no known simple polymer representation for ASEP, and even the question whether or not a slow bond of arbitrarily small strength affects the asymptotic current is open.  

%

\subsection{TASEP and Last Passage Percolation}
\label{s:lpp}
One can map TASEP on $\Z$ into a directed last passage percolation model on $\Z^2$ with i.i.d.\ Exponential weights, and much of the recent advances in understanding of TASEP has come from looking at the corresponding last passage percolation picture. For each vertex $v\in \Z^2$ associate i.i.d.\ weight $\xi_{v}$ distributed as $\mbox{Exp}(1)$. Define $u\preceq v$ if $u$ is co-ordinate wise smaller than $v$ in $\Z^2$. For $u\preceq v$ define the last passage time from $u$ to $v$, denoted $T_{u,v}$ by
$$T_{u,v}:=\max_{\pi} \sum_{v'\in \pi} \xi_{v'}$$
where the maximum is taken over all up/right oriented paths from $u$ to $v$. In particular, let $T_{n}$ denote the passage time from $(0,0)$ to $(n,n)$. One can couple the TASEP with the Exponential directed last passage percolation (DLPP) as follows. For $v=(x,y)\in \Z^2$ let $\xi_{v}$ be the the waiting time for the $(\min(x,y)+1)$-th jump at the site $(x-y)$ (once there is a particle at $x-y$ and the site $x-y+1$ is empty). It is easy to see (see e.g.\ \cite{sas}) that under this coupling $T_n$ is equal to the time it takes for $n+1$ particles to jump out of the origin when TASEP starts with step initial condition, i.e., the time taken by the particle at $-n$ to jump to site $1$. Often, when there is no scope for confusion, we shall denote by $T_{a,b}$ the last passage time from $(0,0)$ to $(a,b)$. In this notation, $T_{n+k,n}$ equals the time taken by the particle at $-n$ to jump to site $k+1$. 

Even when TASEP starts from some arbitrary initial condition, the above coupling can be used to describe jump times as last passage times, but the more general formula involves last passage time from a point to a set rather than last passage time between two points; see Section \ref{s:couple} for more details. 

Observe that in the coupling described above the passage times of the vertices on the line $x-y=i$ describes the weighting times for jumps at site $i$. Using this it is easy to translate the slow bond model to the last passage percolation framework. Indeed, we only need to modify the passage times on the diagonal line $x=y$ by i.i.d.\ $\mbox{Exp}(r)$ variables independently of the passage times of the other sites. We shall use $T^{(r)}$ to denote the last passage times in this model. It is clear from the above discussion that 
$T^{(r)}_{n}$ has the same distribution as $L_n^{(r)}$, where $L_n^{(r)}$ is as in Theorem \ref{t:lln}. For the rest of the paper we shall be working mostly with the last passage picture, implicitly we shall always assume the aforementioned coupling with TASEP to move between TASEP and DLPP even though we might not explicitly mention it every time. We shall see later how statistics from the last passage percolation model can be interpreted in terms of the occupation measures of sites in TASEP, and provide information about invariant measures.

\subsection{The Inputs from Integrable Probability; Tracy-Widom limit, and $n^{2/3}$ Fluctuations and Coalescence of geodesics}
\label{s:int}

The basic idea of \cite{BSS14} and \cite{BSS17++} were to study the environment around maximal paths between two points (henceforth called geodesics) in the Exponential DLPP model at different scales. This is facilitated by the very fine information about the fluctuation of the length of such paths, coming from the integrable probability literature. Following the prediction by Kardar, Parisi and Zhang \cite{KPZ86}, it is believed that passage times $T_{n}$ should have fluctuation of  order $n^{1/3}$ for very general passage time distribution, but this is only known for a handful of models for which some explicit exact calculation is possible using connections to algebraic combinatorics and random matrix theory, these are the so called exactly solvable or integrable models. Starting with the seminal work of Baik, Deift and Johansson \cite{BDJ99} where the $n^{1/3}$ fluctuation and Tracy-Widom scaling limit was established for Poissonian last passage percolation, the KPZ prediction has now been made rigorous for a handful of exactly solvable models; DLPP with exponential passage times among them. The Tracy-Widom scaling limit for exponential DLPP is due to Johansson \cite{Jo99}.

\begin{theorem}[\cite{Jo99}]
\label{t:Jo99}
Let $h>0$ be fixed. Let $v=(0,0)$ and $v_{n}=(n,\lfloor hn \rfloor)$. Then
\begin{equation}
\label{e:twlimitdiscrete}
\dfrac{T_{v,v_n}-(1+\sqrt{h})^2n}{h^{-1/6}(1+\sqrt{h})^{4/3}n^{1/3}} \stackrel{d}{\rightarrow} F_{TW}.
\end{equation}
where the convergence is in distribution and $F_{TW}$ denotes the GUE Tracy-Widom distribution. 
\end{theorem}

GUE Tracy-Widom distribution is a very important distribution in random matrix theory that arises as the scaling limit of largest eigenvalue of GUE matrices; see e.g.\ \cite{BDJ99} for a precise definition of this distribution. For our purposes moderate deviation inequalities for the centred and scaled variable as in the above theorem will be important. Such inequalities can be deduced from the results in \cite{BFP12}, as explained in \cite{BSS14}. We quote the following result from there. 

\begin{theorem}[\cite{BSS14}, Theorem 13.2]
\label{t:moddevdiscrete}
Let $\psi>1$ be fixed. Let $v, v_n$ be as in Theorem \ref{t:Jo99}. Then  there exist constants $N_0=N_0(\psi)$, $t_0=t_0(\psi)$ and $c=c(\psi)$ such that we have for all $n>N_0, t>t_0$ and all $h \in (\frac{1}{\psi}, \psi)$

$$\P[|T_{v,v_{n}}- n(1+\sqrt{h})^{2}|\geq tn^{1/3}]\leq e^{-ct}.$$
\end{theorem}

Theorem \ref{t:moddevdiscrete} provides much information about the geometry and regularity of geodesics in the DLPP model; which was exploited crucially in \cite{BSS14} and \cite{BSS17++}, and these results will be used by us again. Our strategy is to use these to study the geometry of geodesics in the Exponential DLPP as well as on the model with reinforcement on the diagonal. Studying the geometry will enable us to prove facts about the occupation measure of certain sites in the corresponding particle systems, and thus conclude Theorem \ref{t:mainst}. A crucial step in the arguments in this paper would use the coalescence of geodesics starting from vertices that are close to each other. This result follows from Corollary $3.2$ of \cite{BSS17++} with $k=1$ and $R=n$.

\begin{theorem}[\cite{BSS17++}, Corollary 3.2]
\label{pathsmeet}
Let $L,L',m_0>0$ be fixed constants. Let $\Gamma$ be the geodesic from $(0,L)$ to $(n, m_0n+L'n^{2/3})$ and $\Gamma'$ be the geodesic from $(0,-L)$ to $(n,m_0n-L'n^{2/3})$ in the Exponential LPP model. Let $E$ be the event that $\Gamma$ and $\Gamma'$ meet. Then,
\[\P(E)\geq 1-Cn^{-c},\]
for some absolute positive constants $C,c$ (depending only on $L,L',m_0$ but not on $n$).
\end{theorem}

\subsection{Outline of the Proofs}
We provide a sketch of the main arguments in this subsection. The proof of Theorem \ref{t:mainst} has two major components. The first part is devoted to establishing the existence of the limiting distribution of TASEP with a slow bond at the origin started from the step initial conditions. In the second part, we show that the limiting distribution is asymptotically equivalent to a product Bernoulli measures with different densities far away from the origin to the left and to the right. Throughout the rest of the paper we shall work with a fixed $r<1$ and $\varepsilon=\varepsilon(r)$ given by Theorem \ref{t:lln}.

\subsubsection{A Basic Observation Connecting Last Passage Times and Occupation Measures}
The basic observation underlying all the arguments regarding invariant measures in this paper can be simply stated in the following manner. The amount of time the particle starting at $-n$ spends at $0$ is $T_{n,n}-T_{n-1,n}$ (recall our convention that $T_{a,b}$ is the last passage time from $(0,0)$ to $(a,b)$ when there is no scope for confusion). In the same vein, the total time the site $0$ is occupied between $T_{n}$ and $T_{n+k}$ is determined by the pairwise differences $T_{i,i}-T_{i-1,i}$ for $i\in \{n+1,n+2,\ldots , n+k\}$. More generally for states $I=[-b,b]\cap \Z$, where $b\in \N$ and $A\subseteq \{0,1\}^I$, using the coupling between TASEP and last passage percolation, it is not difficult to see that the occupation measure at sites $I$ between $T_n$ to $T_{n+k}$, i.e., $A\mapsto \int_{T_n}^{T_{n+k}}\ind(\eta_t(I)\in A)dt$ is a function of the pairwise differences of the last passage times $T_{\mathbf{0},v}$ for a set of vertices $v$ around the diagonal between $(n,n)$ and $(n+k,n+k)$. All our arguments about the limiting distribution in this paper are motivated by the above basic observation and the idea that the average occupation measure should be close to the limiting distribution as $n$ and $k$ becomes large.

\subsubsection{Convergence to a Limiting Measure}
The idea of showing that such a limit to the average occupation measure exists is as follows. Using Theorem \ref{t:lln} we can show that in the reinforced last passage percolation model, the geodesics (from $\mathbf{0}$ to $\mathbf{n}$, say where $\mathbf{n}:=(n,n)$ and $\mathbf{0}$ denotes the origin) are pinned to the diagonal and the typical fluctuation of the paths away from the diagonal is $O(1)$. This in turn implies that for some fixed $k$ and $n\gg k$, all the geodesics from $\mathbf{0}$ to points near the diagonal between $\mathbf{n}$ and $\mathbf{n+k}$ merge together at some point on the diagonal near $\mathbf{n}$ with overwhelming probability. This indicates that the pairwise difference of the passage times is determined by the individual passage times  near the region on the diagonal between $\mathbf{n}$ and $\mathbf{n+k}$. Using this localisation into disjoint boxes (with independent passage time configuration), the average occupation measure over a large time can be approximated by an average of i.i.d.\ random measures. A law of large numbers ensure that the average occupation measures converge to a measure. To show that the process, started from a step initial condition converges weakly to this distribution requires a comparison between average occupation measure during a random interval with the distribution at a fixed time and this is done via a smoothing argument using local limit theorems.

\subsubsection{Occupation Measures Far Away from Origin}
\label{s:heuristic}
The second part of the the argument, i.e., to show that the limiting measure is asymptotically equivalent to a product Bernoulli measure far away from the origin, is more involved. Let us restrict to the measure far away to the right of origin. Consider a fixed length interval $[k,k+L]$ for $k\gg 1$. From the above discussion it follows that the average occupation measure at some random time interval after a large time will be determined by the pairwise differences $T_{\mathbf{0}, v_i}-T_{\mathbf{0}, v_{j}}$ for the vertices $v_{i},v_{j}$ in some square of bounded size around the vertex $(n+k,n)$ for $n\gg k$. Let us first try to understand the geodesics $\Gamma_{0,v_{i}}$ from $\mathbf{0}$ to $v_{i}$, and in particular the geodesic from $(0,0)$ to $(n+k,n)$.

\begin{figure}[h] 
\centering
\begin{tabular}{cc}
\includegraphics[width=0.4\textwidth]{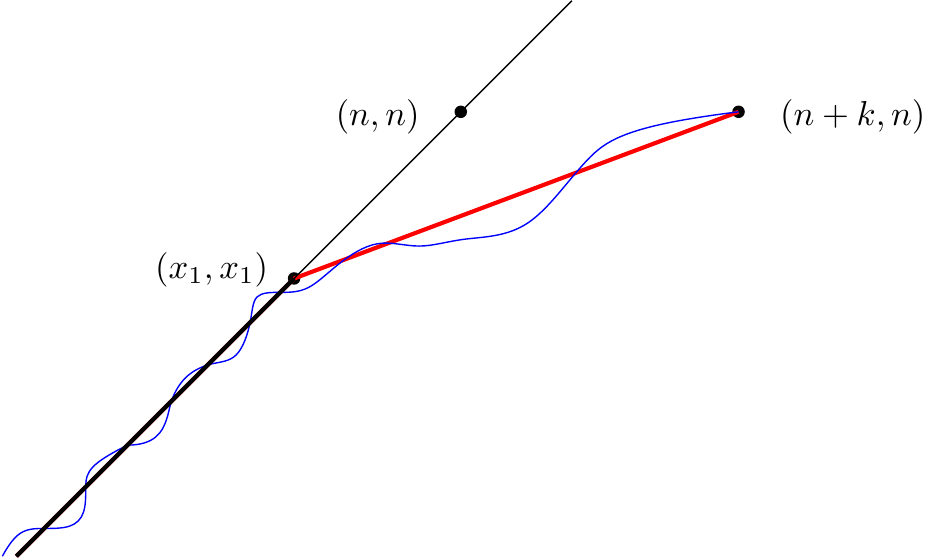} &\quad\quad\quad\quad\includegraphics[width=0.4\textwidth]{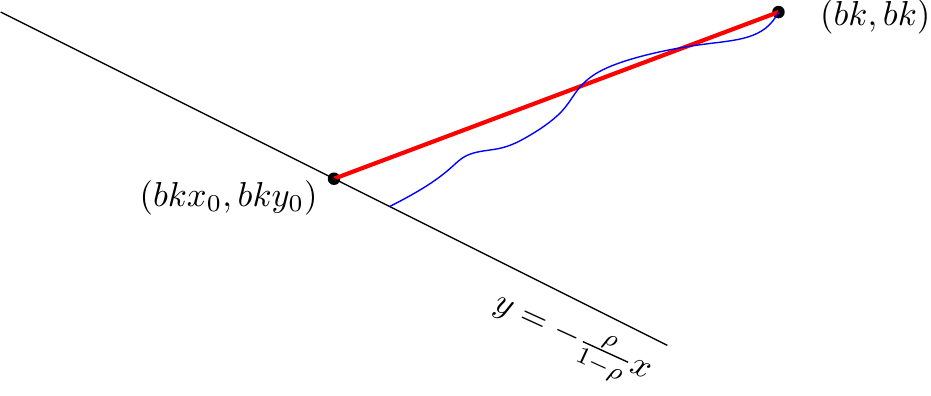} \\
(a) & \quad\quad\quad\quad(b)
\end{tabular}
\caption{ (a) First order behaviour of $\Gamma_{n+k,n}$, this remains pinned to the diagonal until point $(x_1,x_1)$ then follows a path that in the the first order is a straight line to $(n+k,n)$. (b) Point to line geodesic to the line $y=-\frac{\rho}{1-\rho} x$. This corresponds to TASEP with a product $\mbox{Ber}(\rho)$ initial condition. The geodesic in first order looks like a straight line. Our proof will compare the second part of the geodesic $\Gamma_{n+k,n}$ to the point to line geodesic in (b). We choose $\rho$ such that the slopes match.} 
\label{f:couple1}
\end{figure}

It is easy to see that the geodesic should be pinned to the diagonal until about $O(k)$ distance from $(n,n)$ and then should approximately look like a geodesic in the unconstrained model to $(n+k,n)$. The approximate location until which this pinning occurs can be computed by a first order analysis using Theorem \ref{t:moddevdiscrete} and Theorem \ref{t:lln} (Theorem \ref{t:moddevdiscrete} implies that $\E T_{x,y}\approx (\sqrt{x}+\sqrt{y})^2$ in the first order). These estimates yield that the last hitting point of the diagonal should be close to the point $(x_1,x_1)$ where $x_1$ maximises the function $(4+\varepsilon)x+(\sqrt{n+k-x}+\sqrt{n-x})^2$. An easy optimization gives 
\begin{equation}\label{e:defx1}
n-x_1=k\frac{(\sqrt{4+\varepsilon}-\sqrt{\varepsilon})^2}{4\sqrt{\varepsilon(4+\varepsilon)}}.
\end{equation}
Hence, the slope of the line joining $(n+k,n)$ to $(x_1,x_1)$ is 
\begin{equation}
\label{e:slope1}
\left(\frac{\sqrt{4+\varepsilon}-\sqrt{\varepsilon}}{\sqrt{4+\varepsilon}+\sqrt{\varepsilon}}\right)^2,
\end{equation}
and so the unpinned part of the geodesic $\Gamma_{0, (n+k,n)}$ is approximately a geodesic in an an environment of i.i.d.\ Exponentials along a direction with the slope given by the above expression. We shall show the following. \emph {For $k\gg 1$, with probability close to 1 the geodesics from all the vertices on a square of bounded size to  $0$, coalesce before reaching the diagonal near $(x_1,x_1)$}. This follows from Theorem \ref{pathsmeet}. This in turn will imply that for large $k$, on a set with probability close to 1,  that the pairwise differences between those passage time will be locally determined by the i.i.d.\ exponential random environment in some box of size $\ll k$ around the point $(n+k,n)$.  

The final observation above will allow us to couple TASEP with a slow bond together with a stationary TASEP with product $\mbox{Ber}(\rho)$ (for some appropriately chosen $\rho$) stationary distribution, so that with probability close to 1 (under the coupling), the average occupation measure on the interval $[k,k+L]$ (here $L$ is fixed) during some late and large interval of time for the slow bond TASEP, will be equal to the average occupation measure of the interval $[0,L]$ during an interval of same time. Hence the occupation measures will be close in total variation distance. Due to the stationarity of the latter process the latter occupation measure is close to product $\mbox{Ber}(\rho)$, and we shall be done by taking appropriate limits.

\subsubsection{A Coupling with a Stationary TASEP}

Towards constructing the coupling described above we use the correspondence between a stationary TASEP and Exponential last passage percolation. It is well known that jump times in stationary TASEP corresponds to last passage times in point-to-set last passage percolation in i.i.d.\ Exponential environment just as the jump times in TASEP with step initial condition corresponds to point-to-point passage times. See Section \ref{s:couplestat} for a precise definition; roughly the following is true. There exists a random curve $S$ (a function of the realisation of the stationary initial condition), such that the jump times of the stationary TASEP correspond to the last passage time from $S$ to $v$ for vertices $v\in \Z^2$ (naturally the last passage time from $S$ to $v$ means the maximum passage time of all paths that start somewhere in $S$ and end at $v$). It is standard, that for a product $\mbox{Ber}(\rho)$ initial condition the curve $S$ is well approximated by the line $\mathbb{L}$ with the equation 
$$y=-\frac{\rho}{1-\rho}x.$$

\begin{figure}[h] 
\centering
\begin{tabular}{cc}
\includegraphics[width=0.4\textwidth]{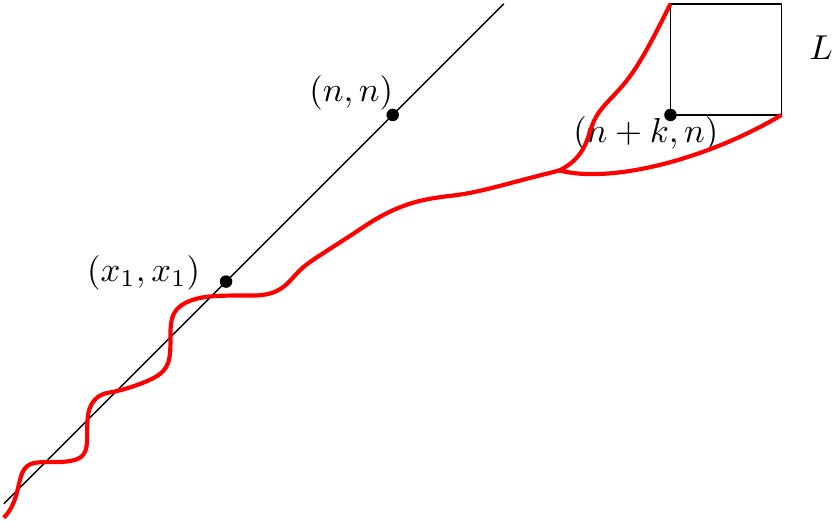} &\quad\quad\quad\quad\includegraphics[width=0.4\textwidth]{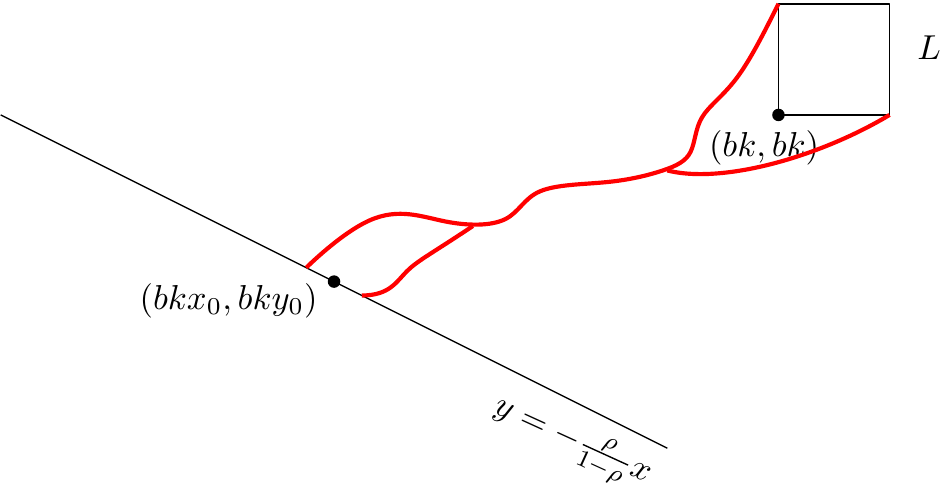} \\
(a) & \quad\quad\quad\quad(b)
\end{tabular}
\caption{We shall construct a coupling between the two systems: (a) TASEP with a slow bond, and (b) Stationary TASEP with product $\mbox{Ber}(\rho)$ initial condition. Roughly the coupling will assign the same individual vertex weights to both the systems upto a translation that takes the marked square in (a) to the marked square in (b). Using coalescence of geodesics, we shall show that on a large probability event for $n\gg k\gg 1$, the pairwise difference of passage times to the points in the marked square, will be determined locally in both the systems, and hence will be identical under the coupling.} 
\label{f:couple2}
\end{figure}

At this point we perform another first order calculation to determine the approximate slope for geodesics from $\mathbb{L}$ to $(n,n)$. Such a geodesic should hit the line $\mathbb{L}$ close to the point $(nx_0, ny_0)$ for $(x_0,y_0)$ which maximizes $(\sqrt{n-nx}+\sqrt{n-ny})^2$ for all $(x,y)\in \L$. An easy calculation gives that the approximate slope of the geodesic is  $\left(\frac{\rho}{1-\rho}\right)^2$. To construct a successful coupling, one needs to match this slope to the one in \eqref{e:slope1}. Solving the equation one gets 
\begin{equation}\label{rho}
\rho=\frac{\sqrt{4+\varepsilon}-\sqrt{\varepsilon}}{2\sqrt{4+\varepsilon}}.
\end{equation}
Observe that $\rho< \frac{1}{2}$ and $\rho(1-\rho)=\frac{1}{4+\varepsilon}$ as expected. Now there is a natural way to couple the two processes. Recall the point $(x_1,x_1)$. Choose $b>0$ and a translation of $\Z^2$ that takes $(bk,bk)$ to $(n+k,n)$ and $(bkx_0, bky_0)$ to $(x_1,x_1)$. Coupling the environment below the diagonal in the LPP corresponding to the slow bond model with the environment in the LPP of stationary TASEP (naturally under the above translation which takes $\mathbb{L}$ to a parallel line passing through $(x_1,x_1)$) and using the coalescence result Theorem \ref{pathsmeet} would give the required result under this coupling. (This is morally correct, even though not technically entirely precise. The formal argument is slightly more involved because it has to take care of the effect of the reinforced diagonals; see Section \ref{s:couple} for details).

\begin{remark}
\label{r:negative}
One can redo the same argument far away to the left of the origin. For an interval $[-k-L,-k]$ with $k\gg 1$ one finds that the appropriate density $\rho'$ for the stationary TASEP to be coupled with the slow bond TASEP is given by
\begin{equation}\label{rhodash}
\rho'=\frac{\sqrt{4+\varepsilon}+\sqrt{\varepsilon}}{2\sqrt{4+\varepsilon}}.
\end{equation}
and $\rho'$ is indeed equal to $1-\rho$ as it should be.
\end{remark}

\subsection{Notations}
For easy reference purpose, let us collect here a number of notations, some of which have already been introduced, that we shall use throughout the remainder of this paper. Define the partial order $\preceq $ on $\Z^2$ by $u = (x,y) \preceq u' = (x',y')$ if $x \leq x'$, and $y \leq y'$. For $\mathbf{a},\mathbf{b}\in \Z^2$ with $\mathbf{a}\preceq  \mathbf{b}$, let $\Gamma^{(r)}_{\mathbf{a},\mathbf{b}}$ denote the geodesic from $\mathbf{a}$ to $\mathbf{b}$ in the reinforced model when the passage times on the diagonal have been changed to i.i.d.\ Exp($r$) variables. Also when $\mathbf{a}=0$, $\Gamma^{(r)}_{\mathbf{0},\mathbf{b}}$ is simply denoted as $\Gamma^{(r)}_{\mathbf{b}}$, and for $\mathbf{b}=(b,b)$, $\Gamma^{(r)}_{\mathbf{b}}$ is denoted simply as $\Gamma^{(r)}_b$. We shall drop the superscript $r$ when there is no scope for confusion. For the usual exponential DLPP, i.e. when $r=1$, we shall abuse the notation and denote the corresponding geodesics by $\Gamma^0_{\mathbf{a},\mathbf{b}},\Gamma^0_{\mathbf{b}},\Gamma^0_b$. We shall also denote by $T_{\mathbf{a},\mathbf{b}}$ (resp.\ $T_{a,b}^0$) the weight of the geodesic $\Gamma_{\mathbf{a},\mathbf{b}}$ (resp.\ $\Gamma^0_{\mathbf{a},\mathbf{b}}$). \footnote{In certain settings we shall work with the following modified passage times without explicitly mentioning so. For an increasing path $\gamma$ from $v_1$ to $v_2$ let us denote the passage time of $\gamma$ by
$$\ell (\gamma)= \sum_{v\in \gamma \setminus \{v_2\}} \xi_{v};$$
Observe that this is a little different from the usual definition of passage time as we exclude the final vertex while adding weights. This is done for convenience as our definition allows $\ell(\gamma)=\ell(\gamma_1)+\ell(\gamma_2)$ where $\gamma$ is the concatenation of $\gamma_1$ and $\gamma_2$. As the difference between the two definitions is minor while considering last passage times between far away points, all our results will be valid for both our and the usual definition of LPP.}

For $u \preceq  u'$ in $\Z^2$, let $\mbox{Box}(u,u')$ denote the rectangle with bottom left corner $u$ and top right corner $u'$. 
For an increasing path $\gamma$ and $\ell\in \Z$, $\gamma(\ell)\in \Z$ will denote the maximum number such that $(\ell,\gamma(\ell))\in \gamma$ and $\gamma^{-1}(\ell)\in \Z$ be the maximum number such that $(\gamma^{-1}(\ell),\ell)\in \gamma$. 

As we shall be working on $\Z^2$, often we use the notation $\llbracket \cdot, \cdot \rrbracket$ for discrete intervals, i.e., $\llbracket a, b \rrbracket$ shall denote $[a,b]\cap \Z$. In the various theorems and lemmas, the values of the constants $C,C',c,c'$ appearing in the bounds change from one line to the next, and will be chosen small or large locally.


\subsection{Organization of the Paper}
The remainder of this paper is organized as follows. Section \ref{s:GeodSlow} develops the geometric properties of the geodesics in the slow bond model, in particular the diffusive fluctuations of the geodesics and localisation and coalescence of geodesics near the diagonal are shown in this section for the reinforced model. Section \ref{s:inv} is devoted to constructing a candidate for the invariant measure in the slow bond TASEP by passing to the limit of average occupation measures. That this measure is the limiting measure of the slow bond TASEP started from step initial condition is shown in Section \ref{s:conv}. Finally, Section \ref{s:couple} deals with the coupling between the slow bond TASEP and the stationary TASEP that ultimately leads to the proof of Theorem \ref{t:mainst}. We finish off the paper in subsection \ref{s:otherrho} by providing a sketch of the argument for Theorem \ref{t:otherrho}. The proofs of a few technical lemmas used in Sections \ref{s:conv} and \ref{s:couple} have been relegated to Appendix A (Section \ref{s:techlem}). Additionally, a Central Limit Theorem for the passage times in slow bond TASEP is provided in Appendix B (Section \ref{s:clt}) as this has not been not directly used in the paper.

\subsection*{Acknowledgements}
The authors thank Thomas Liggett for telling them about the problem and are grateful to Thomas Liggett and Maury Bramson for discussing the background and useful comments on an earlier version of this paper. RB  would also like to thank Shirshendu Ganguly and Vladas Sidoravicius for useful discussions.

\section{Geodesics in Slow Bond Model}
\label{s:GeodSlow}
Consider the Exponential last passage percolation model corresponding to TASEP with a slow bond at the origin that rings at rate $r<1$. We shall work with a fixed $r<1$ throughout the section and $\varepsilon$ will be as in Theorem \ref{t:lln}. We shall refer to this as the slow bond model when there is no scope for confusion. It follows easily by comparing \eqref{e:lln} and Theorem \ref{t:lln} that in this model the geodesic is \emph{pinned} to the diagonal, i.e., the expected number of times the geodesic $\Gamma_{n}$ between $\mathbf{0}=(0,0)$ to $\mathbf{n}=(n,n)$ (as there is no scope of confusion we shall suppress the dependence of $\Gamma$ on $r$) hits the reinforced diagonal line is linear in $n$. In this section we establish stronger geometric properties of those geodesics. Indeed we shall show that the typical distance between two consecutive points on $\Gamma_{n}$ that are on the diagonal is $O(1)$ and also the transversal distance of $\Gamma_{n}$ from the diagonal at a typical point is also $O(1)$. We begin with the following easy lemma. 

\begin{lemma}
\label{i:davoid}
There exists absolute positive constants $m_0,c$ (depending only on $r$) such that for any $m\in \N$, $m\geq m_0$, the probability that $\Gamma_m$ does not touch the diagonal between $\mathbf{0}$ and $\mathbf{m}$ is at most
$e^{-cm^{2/3}}$.
\end{lemma}

The proof follows easily by comparing the lengths of the geodesics in the reinforced and unreinforced environments.

\begin{proof} Consider the coupling between the slow bond model and the (unreinforced) DLPP where the passage times at all vertices not on the diagonal are same, and those on the diagonal are replaced by i.i.d.\ Exp($r$) variables independent of all other passage times. Then it is easy to see that if $\Gamma_m$ avoids the diagonal between $(0,0)$ and $(m,m)$, then it is the maximal path in the unreinforced environment between $(0,0)$ and $(m,m)$ that never touches the diagonal in between, and hence its length is at most the length of the geodesic in the unreinforced environment. Hence,
\[\P(\Gamma_{m} \mbox{ avoids diagonal})\leq \P(T_m\leq T_m^0)\leq \P\left(T^0_{m}>(4+\frac{\varepsilon}{2})m\right)+\P\left(T_{m}<(4+\frac{\varepsilon}{2})m\right).\]
Now by moderate deviation estimate in Theorem \ref{t:moddevdiscrete}, for $m\geq m_0$,
\[\P\left(T^0_{m}>(4+\frac{\varepsilon}{2})m\right) \leq e^{-c_2m^{2/3}},\]
where $c_2$ is a constant depending only on $\varepsilon$.
In order to bound the probability that the geodesic in the reinforced environment is not too short, first choose $M$ large enough so that $\E(\frac{T_M}{M})>4+\frac{3\varepsilon}{4}$. Then because of super-additivity of the path lengths, $\frac{T_{nM}}{nM}\geq \frac{X_1+\ldots +X_n}{n}=:\bar{X}$ where $X_i:=\frac{T_{(i-1)\mathbf{M},i\mathbf{M}}}{M}$ are i.i.d.\ random variables, each having the same distribution as that of $X_1=\frac{T_M}{M}$. Let $[m/M]=n$, and $m$ is large enough so that $\frac{4+\varepsilon/2}{n}<\varepsilon/8$, then
\[\P\left(T_m<(4+\frac{\varepsilon}{2})m\right)\leq \P\left(\frac{T_{nM}}{nM}<(4+\frac{\varepsilon}{2})\frac{n+1}{n}\right)<\P\left(\left|\bar{X}-\E X_1\right|>\frac{\varepsilon}{8}\right)\leq e^{-c'n}\leq e^{-cm}.\]
where $c,c'$ are constants depending only $\varepsilon$ and $r$. The last inequality follows as it is easy too see that for a fixed $M$, ${T_M}\preceq _{ST} \mbox{Gamma}(M^2,r)$  which has exponential tails where $\preceq _{ST}$ denotes stochastic domination.
\end{proof}

We remark that the exponent here is not optimal. One can prove an upper bound of $e^{-cm}$ by using large deviation estimates from \cite{Jo99} instead of Theorem \ref{t:moddevdiscrete}, but this is sufficient for our purposes.


The following proposition controls the transversal fluctuation of the geodesics $\Gamma_{n}$. Recall that for $\ell\in \Z$, $\Gamma(\ell)\in \Z$ is the maximum number such that $(\ell,\Gamma(\ell))\in \Gamma$ and $\Gamma^{-1}(\ell)\in \Z$ be the maximum number such that $(\Gamma^{-1}(\ell),\ell)\in \Gamma$ .

\begin{proposition}
\label{deviationfromdiagonal} 
For $h\in  \llbracket 0,n \rrbracket$, we have for all $n \geq m>m_0$, for some absolute positive constants $m_0,c$,
$$\P(|\Gamma_n(h)-h| \geq m) \leq e^{-cm^{1/2}}, \mbox{ and}$$
\[\P(|\Gamma_{n}^{-1}(h)-h|>m)\leq e^{-cm^{1/2}}.\]
\end{proposition}

\begin{proof}
For the purpose of this proof we drop the subscript $n$ from $\Gamma_{n}$. First note that if $|\Gamma(h)-h|>m$, then $h\geq m$ or $n-h\geq m$. If $B$ is the event that the geodesic $\Gamma$ hits the diagonal at $(i,i)$ and returns to the diagonal again at $(j,j)$ with $|j-i|\geq s$, then applying previous lemma \ref{i:davoid},
\[\P(B)\leq \sum_{j=s}^\infty e^{-cj^{2/3}}\leq e^{-c's^{7/12}}.\]
Hence if $\Gamma(h)>h+m$, by summing up over all positions where $\Gamma$ touches the diagonal for the last time before $(h,h)$, one has,
\[\P[\Gamma(h)-h>m]\leq \sum_{i=0}^h e^{-c'(i+m)^{7/12}}\leq \sum_{i=m}^\infty e^{-c'i^{7/12}}\leq e^{-cm^{1/2}}.\]
Similar arguments work for the events $\{\Gamma(h)<h-m\}$ and $\{|\Gamma^{-1}(h)-h|>m\}$. 
\end{proof}

Notice that it is not hard to establish using similar arguments that for pairs of points not far away from the diagonal, the geodesic between them also has $O(1)$ transversal fluctuation from the diagonal at a typical point. We state this below without a proof.
\begin{corollary} \label{c:devdiag}
Let $h\in \llbracket 0,n\rrbracket$ and $t=o(h)$ and $t'=o(n-h)$ and $\Gamma=\Gamma_{(0,t),(n,n+t')}$. Then there exist absolute positive constants $n_0,h_0,m_0,c$ such that for all $n\geq n_0,h\geq h_0, m\geq m_0$,
\[\P[|\Gamma(h)-h| \geq m] \leq e^{-cm^{1/2}}, \mbox{ and}\]
\[\P(|\Gamma^{-1}(h)-h|>m)\leq e^{-cm^{1/2}}.\].
\end{corollary}

Our next result will establish something stronger. We shall show that typically geodesic between every pair of points, one of which is close to $\mathbf{0}$ and the other close to $\mathbf{m}$, meet the diagonal \emph{simultaneously}. 

\begin{theorem}
\label{meetondiagonal} 
Fix $0<\alpha <1$. Let $L_1$ be the line segment joining $(0,-m^\alpha)$ to $(0,m^{\alpha})$. Similarly $L_2$ be the line segment joining $(m,m-m^{\alpha})$ to $(m,m+m^{\alpha})$. Let $\ce$ denote the event that there exists $u\in \llbracket 0,m \rrbracket$ such that $(u,u)\in \Gamma_{a,b}$ for all $a\in L_1\cap \Z^2, b\in L_2\cap \Z^2$. Then there exist some absolute positive constants $m_0,c$ such that for all $m\geq m_0$, $\P(\ce)\geq 1-e^{-cm^{\ell}}$ where $\ell=\min\{\frac{1-\alpha}{2},\frac{\alpha}{2}\}$.
\end{theorem}

We emphasize again that in this Theorem \ref{meetondiagonal} as well as in the preceding lemmas, we have been very liberal about the exponents, and have not always attempted to find the best possible exponents in the bounds, as long as they suffice for our purpose. 

We shall need a few lemmas to prove Theorem \ref{meetondiagonal}. The following lemma is basic and was stated in \cite{BSS14}, we restate it here without proof.

\begin{lemma}[\cite{BSS14}, Lemma 11.2, Polymer Ordering:]
\label{l:porder}
Consider points $a=(a_1,a_2),a'=(a_1,a_3),b=(b_1,b_2),b'=(b_1,b_3)$ such that $a_1 < b_1$ and $a_2\leq a_3\leq b_2\leq b_3$. Then we have $\Gamma_{a,b}(x)\leq \Gamma_{a',b'}(x)$ for all $x\in [a_1,b_1]$.
\end{lemma}

The next lemma shows that two geodesics between pairs of points not far from the diagonal have a positive probability to pass through the midpoint of the diagonal. Define $\alpha'=\frac{\alpha+1}{2}$. Clearly $\alpha<\alpha'<1$.
\begin{lemma}
\label{l:diagpos} Let $a_1=m^{\alpha'}$, $\Gamma_1=\Gamma_{(0,-m^\alpha),(a_1,a_1-m^{\alpha})}$, $\Gamma_2=\Gamma_{(0,m^\alpha),(a_1,a_1+m^{\alpha})}$ and $v:=(\frac{a_1}{2},\frac{a_1}{2})$. Let $F_1$ be the event that $v\in \Gamma_1\cap \Gamma_2$. Then there exists some absolute positive constant $\delta$ such that $\P(F_1)\geq \delta$.
\end{lemma}
The idea is as follows. There is a positive probability that the exponential random variable $\xi_v$ at $v$ is large, and all other random variables that lie in a large but constant sized box around $v$ are small. As Corollary \ref{c:devdiag} says that $\Gamma_1$ and $\Gamma_2$ are very likely to be in close proximity to $v$, they have a positive probability to pass through $v$. Formally, we do the following.

\begin{proof}[Proof of Lemma \ref{l:diagpos}] Define 
\[\cB_1=\mbox{Box}((\frac{a_1}{2},\frac{a_1}{2}-C),(\frac{a_1}{2}+C,\frac{a_1}{2}));\]
and 
\[\cB_2=\mbox{Box}((\frac{a_1}{2}-C,\frac{a_1}{2}),(\frac{a_1}{2},\frac{a_1}{2}+C))\] 
as squares of side length $C$ with a common vertex $v$. Here $C$ is an absolute constant to be chosen appropriately later. See Figure \ref{f:dcoalesce} (a). Define the following events
\[D_1=\left\{\sum_{x\in \cB_1,x\neq v} \xi_x<2C^2,\sum_{x\in \cB_2,x\neq v} \xi_x<2C^2\right\};\]
\[D_2=\left\{|\Gamma_j(v)-v|\leq C,|\Gamma^{-1}_j(v)-v|\leq C \mbox{ for } j=1,2\right\}.\]
Note $D_1$ is the intersection of two independent high probability events as sum of $C^2-1$ many i.i.d.\ exponential random variables is less than $2C^2$ with high probability for large enough $C$. Also from Corollary \ref{c:devdiag}, it follows that $D_2$ is the intersection of two events with high probability, hence can be made to occur with arbitrarily high probability by choosing $C$ large (note that $m^\alpha=o(a_1)$). Hence choose $C$ large enough so that $\P(D_1)\geq \frac{3}{4}$ and $\P(D_2)\geq \frac{3}{4}$, so that $\P(D_1\cap D_2)\geq \frac{1}{2}$.

Notice that both the events $\{\xi_v>2C^2\}$ and $D_2$ are increasing in the value at $v$ given the configuration on $\R^2\setminus\{v\}$. Hence by the FKG inequality and the fact that $\{\xi_v>2C^2\}$ and $D_1$ are independent, it follows that,
\[\P(\xi_v>2C^2|D_1\cap D_2)\geq \P(\xi_v>2C^2|D_1)=\P(\xi_v>2C^2)=e^{-2rC^2}.\]
Hence
\[\P\left(\left\{\xi_v>2C^2\right\}\cap D_1\cap D_2\right)\geq \frac{e^{-2rC^2}}{2}=:\delta.\]

We claim that on $\{\xi_v>2C^2\}\cap D_1\cap D_2$, both $\Gamma_1$ and $\Gamma_2$ pass through $v$. To see this, define $z_1$ to be the point where $\Gamma_1$ enters one of the boxes $\cB_1$ or $\cB_2$ and $w_1$ denote the point where it leaves the box, similarly define $z_2$ and $w_2$ as the box entry and exit of $\Gamma_2$. Since $D_2$ holds, we can join $z_1$ and $w_1$ to $v$ by line segments and get an alternate increasing path that equals $\Gamma_1$ everywhere else, and inside the box it goes from $z_1$ to $v$ to $w_1$ in straight lines. Call this new path $\Gamma'_1$ (see Figure \ref{f:dcoalesce} (a)). Because of the events $\{\xi_v>2C^2\}$ and $D_1$, the weight of $\Gamma'_1$ is more than that of $\Gamma_1$, unless $\Gamma'_1=\Gamma_1$. Thus on $\{\xi_v>2C^2\}\cap D_1\cap D_2$, $\Gamma_1$ passes through $v$. A similar argument applies to $\Gamma_2$. Hence 
\[\P(F_1)=\P(v\in \Gamma_1\cap \Gamma_2)\geq \P(\{\xi_v>2C^2\}\cap D_1\cap D_2)\geq \delta>0.\] 
\end{proof}

\begin{figure}[h] 
\centering
\begin{tabular}{cc}
\includegraphics[width=0.4\textwidth]{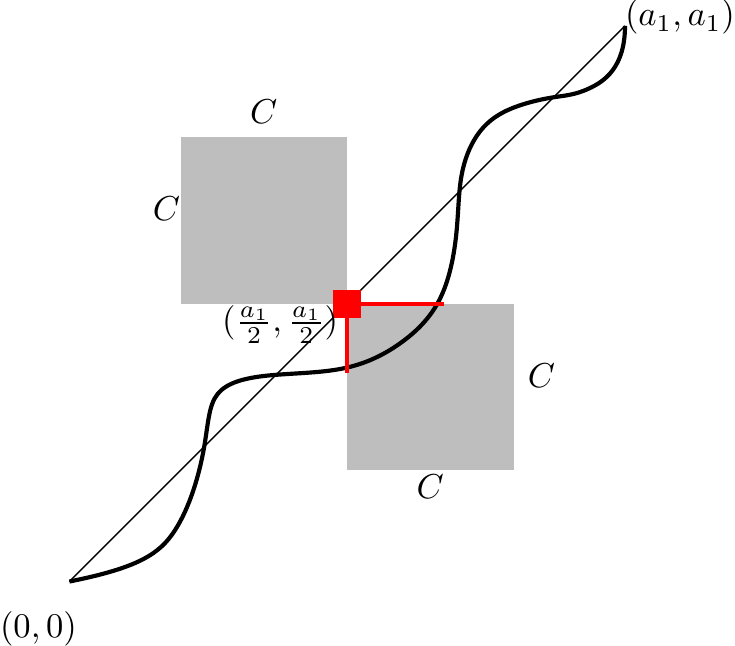} &\quad\quad\quad\quad\includegraphics[width=0.4\textwidth]{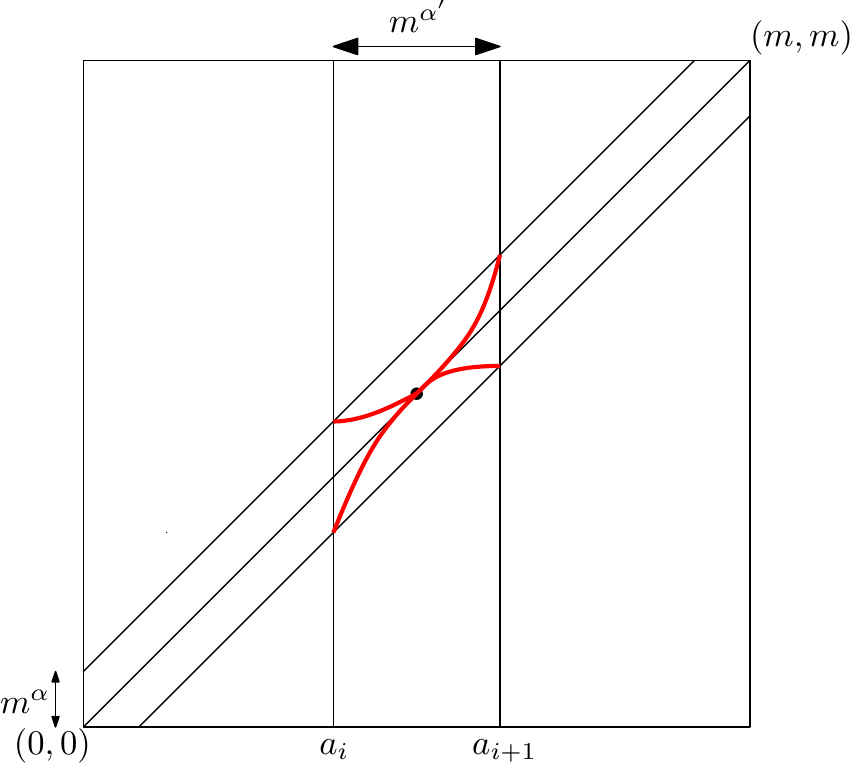} \\
(a) & \quad\quad\quad\quad(b)
\end{tabular}
\caption{(a) Proof of Lemma \ref{l:diagpos}: on the positive probability event that the path passes close to $(\frac{a_1}{2}, \frac{a_1}{2})$, the weight of that vertex is abnormally large, and that of the nearby vertices are typical the path passes through $(\frac{a_1}{2}, \frac{a_1}{2})$. (b) Proof of Theorem \ref{meetondiagonal}: there is a positive probability that all the paths between $a_i$ and $a_{i+1}$ coalesce at the midpoint on the diagonal.} 
\label{f:dcoalesce}
\end{figure}

Now we prove Theorem \ref{meetondiagonal}. Recall that $\alpha'=\frac{\alpha+1}{2}$. The idea is to break up the diagonal into intervals of lengths $m^{\alpha'}$. Because of Proposition \ref{deviationfromdiagonal}, we know that all the geodesics stay close to the diagonal at each of the endpoints of these intervals. The above Lemma \ref{l:diagpos} together with polymer ordering ensures that all these paths meet at the midpoints of each interval with positive probability; see Figure \ref{f:dcoalesce} (b). Because of independence in each interval, the theorem follows. This kind of argument is very crucial and has been repeated throughout the paper.
\begin{proof}[Proof of Theorem \ref{meetondiagonal}] We formalize the above idea. Let $\tilde{\Gamma}_1=\Gamma_{(0,m^{-\alpha}),(m,m-m^{-\alpha})}$ and $\tilde{\Gamma}_2=\Gamma_{(0,m^{\alpha}),(m,m+m^{\alpha})}$. Notice that, if there exists $u\in \llbracket 0,m \rrbracket$ such that $u\in \tilde{\Gamma}_1\cap \tilde{\Gamma}_2$, then because of polymer ordering as stated in Lemma \ref{l:porder}, $u\in \Gamma_{a,b}$ for all $a\in L_1\cap \Z, b\in L_2\cap \Z$. Define $n=m^{1-\alpha'}=m^{\frac{1-\alpha}{2}}$ and
\[a_i=im^{\alpha'} \mbox{ for } i=0,1,2,\ldots,n.\]
Define the event $A$ as
\[A=\left\{|\tilde{\Gamma}_j(a_i)-a_i|\leq m^{\alpha} \mbox{ for all } j=1,2; i=0,1,2,\ldots, n\right\}.\]
Then from Proposition \ref{deviationfromdiagonal} it follows by taking union bound,
\[\P(A)\geq 1-2m^{\frac{1-\alpha}{2}}e^{-cm^{\frac{\alpha}{2}}}\geq 1-e^{-c'm^{\frac{\alpha}{2}}}.\]
Also for $i=1,2,\ldots,n$ define the events $F_i$ as
\[F_i=\left\{\left(\frac{a_{i-1}+a_i}{2},\frac{a_{i-1}+a_i}{2}\right)\in \Gamma_{(a_{i-1},a_{i-1}+m^{\alpha}),(a_{i},a_{i}+m^{\alpha})}\cap \Gamma_{(a_{i-1},a_{i-1}-m^{\alpha}),(a_{i},a_{i}-m^{\alpha})}\right\}.\]

Again due to polymer ordering, it is easy to see that,
\[\ce \supseteq A\cap \left\{\bigcup_{i=1}^{n}F_i\right\} \]
As the $F_i$'s are i.i.d.\ and $\P(F_1)\geq \delta$ by Lemma \ref{l:diagpos}, hence,
\[\P(\ce^c)\leq \P(\cap_iF_i^c)+\P(A^c)\leq \prod_i^n\P(F_i^c)+e^{-c'm^{\frac{\alpha}{2}}}\leq (1-\delta)^{n}+e^{-c'm^{\frac{\alpha}{2}}}\leq e^{-cm^{\ell}},\]
where $\ell=\min\{\frac{1-\alpha}{2},\frac{\alpha}{2}\}$.
\end{proof}

The following corollary follows easily from Theorem \ref{meetondiagonal}. It says that a collection of geodesics whose starting points are close to each other and so are their endpoints, has a high probability to meet the diagonal simultaneously. 
 
\begin{corollary} 
\label{c:meetdiagpar}
Fix $0<\alpha<1$ and $K>0$. Let $U_1$ be the parallelogram whose four vertices are $(0,m^\alpha)$, $(0,-m^{\alpha})$, $(m^K,m^K+m^\alpha)$, $(m^K,m^K-m^\alpha)$. Similarly, define $U_2$ as the parallelogram with vertices $(m^K+m,m^K+m+m^\alpha)$, $(m^K+m,m^K+m-m^\alpha)$, $(2m^K+m,2m^K+m+m^\alpha)$ and $(2m^K+m,2m^K+m-m^\alpha)$. Let $\cA$ denote the event that there exists $u\in \llbracket m^K,m^K+m\rrbracket$ such that $(u,u)\in \Gamma_{a,b}$ for all $a\in U_1\cap \Z^2, b\in U_2\cap \Z^2$. Then there exists constants $m_0,c$ depending only on $K$, such that for all $m\geq m_0$, $\P(\cA)\geq 1-e^{-cm^\l}$ where $\l=\min\{\frac{1-\alpha}{2},\frac{\alpha}{2}\}$.
\end{corollary}
\begin{proof}
Let $L_1$ be the line segment joining $(m^K+\frac{m}{3},m^K+\frac{m}{3}+m^\alpha)$ and $(m^K+\frac{m}{3},m^K+\frac{m}{3}-m^\alpha)$. Similarly let $L_2$ be the line segment joining $(m^K+\frac{2m}{3},m^K+\frac{2m}{3}+m^\alpha)$ and $(m^K+\frac{2m}{3},m^K+\frac{2m}{3}-m^\alpha)$. Let $\cB$ denote the event for all $a\in U_1\cap \Z^2, b\in U_2\cap \Z^2$, $\Gamma_{a,b}(m^K+\frac{m}{3})\in L_1$ and $\Gamma_{a,b}(m^K+\frac{2m}{3})\in L_2$. Using Corollary \ref{c:devdiag} and union bound, it is easy to see that, 
\[\P(\cB)\geq 1-4m^{2K}e^{-cm^{\alpha/2}}\geq 1-e^{-c'm^{\alpha/2}}.\] 
Hence applying Theorem \ref{meetondiagonal} to all geodesics from $L_1$ to $L_2$, one has the result.
\end{proof}

Similarly, the following corollary is immediate. We omit the proof. 

\begin{corollary}
\label{c:twopathsmeet}
Fix $0<\alpha<1$. Let $0<a<b<L$ such that $|b-a|=m$ and $a\geq n,L-b\geq n$ and $n\geq m$. Let $E_1$ be the line segment joining $(0,-n^\alpha)$ and $(0,n^\alpha)$. Let $E_2$ be the line segment joining $(L,L-n^\alpha)$ and $(L,L+n^\alpha)$. Let $B_1$ be the line segment joining $(a,a-m^\alpha)$ and $(a,a+m^\alpha)$, and $B_2$ be the segment joining $(b,b-m^\alpha)$ and $(b,b+m^\alpha)$. Let $\ce$ be the event that there exists $u\in \llbracket a,b \rrbracket$ such that $(u,u)\in \Gamma_{i,j}$ for all $i\in E_1, j\in E_2$ and all $i\in B_1,j\in B_2$. Then there exists absolute positive constant $c$ such that $\P(\ce)\geq 1-e^{-cm^{\l}}$ where $\l=\min \{\frac{1-\alpha}{2},\frac{\alpha}{2}\}$.
\end{corollary}

\subsection{Subdiffusive Fluctuations of the Last Passage Time}
Unlike the TASEP where the passage times $L_n$ is of the order $n^{1/3}$, in presence of a slow bond the passage times $T_n$ (note we suppress the dependence on $r$) show diffusive fluctuation. This is a consequence of the path getting pinned to the diagonal at a constant rate; and using Theorem \ref{meetondiagonal} one can argue that $T_n$ can be approximated by partial sums of a stationary process. Using this, and the mixing properties guaranteed by Theorem \ref{meetondiagonal}, it is possible to prove a central limit theorem for $T_{n}$. Although such a result is interesting, it is not crucial for our purposes in this paper. We shall often want to compare best paths in the reinforced environment (i.e., the slow bond model with the diagonals boosted) with paths that do not use the diagonal. Typically the paths that use the diagonal will be larger, and to quantify this we would need concentration bounds for $|T_{n}-(4+\varepsilon)n|$. This will be done in two steps (a) control on the difference between $\E T_n$ and $(4+\varepsilon)n$ and (b) concentration of $T_n$ around its mean. A proof of the central limit theorem is provided in Section \ref{s:clt}.


We first start with the following lemma.
\begin{lemma}
\label{l:expdiff}
There exists an absolute constant $K>0$ such that 
$$ \E T_{n} \geq (4+\varepsilon)n -K.$$
\end{lemma}
Note that due to superadditivity, $\E T_n\leq (4+\varepsilon)n$ always holds. The main idea in the proof of this lemma is that if $0<a<b<L$, then since the geodesic $\Gamma_L$ is close to the diagonal at $(a,a)$ and $(b,b)$, the part of $\Gamma_L$ falling between the lines $x=a$ and $x=b$ is close in length to that of the geodesic $\Gamma_{(a,a),(b,b)}$. This gives a way to compare geodesics between intervals of different lengths.
\begin{proof}
For any $n \in \N$ and $0<a<b<n$, define $T_n(a,b)$ as the weight of the part of the geodesic $\Gamma_n$ that lies between the vertical lines $x=a$ and $x=b$. That is, $T_n(a,b)=\sum_{(x_i,y_i)\in \Gamma_n: a\leq x_i\leq b} \xi_{(x_i,y_i)}$.
 Then we claim that there exists absolute positive constants $M_0,K$ such that for any $M\geq M_0$, any $n \in \N$ and any $i\in \llbracket 0,M-1\rrbracket$, 
\[|\E(T_{nM}(in,(i+1)n))-\E(T_{i\mathbf{n},(i+1)\mathbf{n}})|\leq K.\]
 To see this fix any $m\in \N$, and define the event $E$ that $\Gamma_{nM}$ and $\Gamma_{i\mathbf{n},(i+1)\mathbf{n}}$ meet together on the diagonal between $\llbracket in,in+m^{1/3} \rrbracket$ and again between $\llbracket(i+1)n-m^{1/3},(i+1)n\rrbracket$, and let $F$ be the event that $\Gamma_{nM}$ passes through the line segment joining $(in,in-m^{1/6})$ and $(in,in+m^{1/6})$, and again through the line segment joining $((i+1)n,(i+1)n-m^{1/6})$ and $((i+1)n,(i+1)n+m^{1/6})$. Let $Y\sim \mbox{Gamma}(2m^{2/3},r)$ is the sum of $2m^{2/3}$ many i.i.d.\ $\mbox{Exp}(r)$ random variables. Then using Proposition \ref{deviationfromdiagonal} and Theorem \ref{meetondiagonal},
\begin{eqnarray*}
&&\P(|T_{nM}(in,(i+1)n)-T_{i\mathbf{n},(i+1)\mathbf{n}}|\geq m)\\
&\leq& \P\left(\left\{|T_{nM}(in,(i+1)n)-T_{i\mathbf{n},(i+1)\mathbf{n}}|\geq m\right\}\cap E\cap F\right)+\P(E^c\cap F)+\P(F^c)\\
&\leq& 2\P(Y\geq m)+C'e^{-cm^{1/12}}+2C'e^{-cm^{1/12}}\\
&\leq & 2C'e^{-cm}+C'e^{-cm^{1/12}}+C'e^{-cm^{1/12}}\leq C''e^{-c'm^{1/12}}.
\end{eqnarray*}
Hence, summing over all $m\in \N$, we have, for all $n,M,i$, there exists some absolute positive constant $K$ such that,
\begin{equation}
\label{e:diffi}
|\E(T_{nM}(in,(i+1)n))-\E(T_{i\mathbf{n},(i+1)\mathbf{n}})|\leq \E|T_{nM}(in,(i+1)n)-T_{i\mathbf{n},(i+1)\mathbf{n}}|\leq K.
\end{equation}
As $T_{i\mathbf{n},(i+1)\mathbf{n}}\overset{d}{=}T_n$, hence, adding up \eqref{e:diffi} over all $i\in \llbracket 0,M-1\rrbracket$, 
\[|\E(T_{nM})-M\E(T_n)|\leq KM.\]
That is,
\[\frac{\E(T_{nM})}{nM}\leq \frac{M\E(T_n)}{nM}+\frac{KM}{nM}=\frac{\E T_n}{n}+\frac{K}{n}.\]
Hence, keeping $n$ fixed, and taking $M\rightarrow \infty$, 
\[4+\varepsilon\leq \frac{\E T_n}{n}+\frac{K}{n}.\]
Hence, for all $n$, $\E T_n\geq (4+\varepsilon)n-K$.
\end{proof}

For the second ingredient, while it is possible to prove a concentration at scale $n^{1/2}$, a tail bound at scale $n^{1/2+o(1)}$ is more standard and much easier to prove. We state, without proof the following result which can be proved using standard Martingale techniques with some truncation (cf.\ the proof of  Theorem $3.11$ in \cite{Kesten}). This will be sufficient for our purpose.  
%
%

\begin{lemma}\label{l:Tndev}
Fix $\delta>0$. Then there exist absolute positive constants $C',c$ such that
\[\P(|T_n-\E T_n|\geq n^{1/2+\delta})\leq C'e^{-cn^{\delta/2}}.\]
\end{lemma}

Lemma \ref{l:expdiff} and Lemma \ref{l:Tndev} imply the following proposition that will be useful later. As discussed earlier in the introduction, in order to look at the limiting distribution away from the origin, we will have to consider geodesics from $(0,0)$ to $(n+k,n)$ for $k$ small compared to $n$. And the geodesic from $(0,0)$ to $(n+k,n)$ is expected to hit the diagonal for the last time near the point $(x_1,x_1)$, such that $x_1$ maximizes the quantity $r(x):=(4+\varepsilon)x+(\sqrt{n+k-x}+\sqrt{n-x})^2.$ 
This is made precise in the following proposition. As before, we have not been very strict about the correct order of the exponents here.

\begin{proposition}
\label{reinforcedx1}
Let $n\geq k^{8/7}$, and $\Gamma_{n+k,n}$ be the geodesic from $(0,0)$ to $(n+k,n)$ in the reinforced environment. Let $(X,X)$ be the last point on the diagonal that lies on $\Gamma_{n+k,n}$. Let $x_1\in \llbracket 0,n\rrbracket$ be the point that maximizes the quantity $r(x)$ above. Then there exist constants $C',c$ such that  
\[\P(|X-x_1|\geq k^{3/5})\leq C'e^{-ck^{1/20}}.\] 

\end{proposition}
 
\begin{proof}
Let $v_1=n-k^{8/7}$ and $\Gamma_1$ be the geodesic from $v=(v_1,v_1)$ to $(n+k,n)$. Let $\ce$ be the event that there exists $u\in \llbracket n-\frac{k^{8/7}}{2},n\rrbracket$ such that $(u,u)\in \Gamma_1\cap\Gamma_{n+k,n}$. Then from Corollary \ref{c:twopathsmeet}, $\P(\ce)\geq 1-e^{-c'k^{1/14}}$. As $\Gamma_1$ and $\Gamma_{n+k,n}$ have the last endpoint common, hence once they meet they coincide till $(n+k,n)$. Thus on $\ce$, the last point on the diagonal for both $\Gamma_1$ and $\Gamma_{n+k,n}$ are same. Hence enough to find the point where $\Gamma_1$ last meets the diagonal.

To this end, first note that from \eqref{e:defx1} in the introduction it follows that $n-x_1=ck$ for some constant $c$. Let $\Gamma_2$ be the union of the two geodesics $\Gamma_{2,1}$ from $v$ to $(x_1,x_1)$ and the geodesic $\Gamma_{2,2}$
from $(x_1,x_1)$ to $(n+k,n)$ that avoids the diagonal. Let $A$ be the event that $\Gamma_1$ touches the diagonal for the last time at some point $(x_2,x_2)$ with $x_2\in \llbracket n-\frac{k^{8/7}}{2},n\rrbracket$ and $|x_2-x_1|\geq k^{3/5}$. Then
\[\P(|X-x_1|\geq k^{3/5})\leq \P(A)+\P(\ce^c)\leq \P(A)+e^{-c'k^{1/14}}.\]
 Let $H=\{(x,y)\in \R^2: y\geq x\}$ denote the region in $\R^2$ that lies on or above the diagonal line $x=y$, and for $z_1,z_2\in \R^2$ with $z_2\geq z_1$ and $z_1,z_2\in \bar{(H^c)}$, let $T_{z_1,z_2}^H$ denote the weight of the geodesic from vertex $z_1$ to $z_2$ that does not pass through the region $H$ (except possibly at the endpoints). Then clearly
\[A\subseteq \bigcup_{x_2: |x_2-x_1|\geq k^{3/5}} \left\{T_{v,(x_2,x_2)}+T^H_{(x_2,x_2),(n+k,n)}\geq T_{v,(x_1,x_1)}+T_{(x_1,x_1),(n+k,n)}^H\right\}.\] 
Calculating expectations using Lemma \ref{l:expdiff}, we get for any such $x_2$, there exists some constant $\alpha$ such that
\[\E(T_{v,(x_2,x_2)})+\E(T^0_{(x_2,x_2),(n+k,n)})\leq \E(T_{v,(x_1,x_1)})+\E(T^0_{(x_1,x_1),(n+k,n)})-\alpha k^{3/5}.\]

Since $|v_1-x_1|\sim k^{8/7}$, by Lemma \ref{l:Tndev} it follows that $\P(|T_{v,(x_1,x_1)}-\E(T_{v,(x_1,x_1)})|\geq \frac{\alpha}{4}k^{3/5})\leq C's^{-ck^{1/20}}$. Also applying Proposition $12.2$ from \cite{BSS14}, one immediately gets that 
$$\P(|T_{(x_1,x_1),(n+k,n)}^H-\E(T^0_{(x_1,x_1),(n+k,n)})|\geq \frac{\alpha}{4}k^{3/5})\leq e^{-ck^{4/45}}.$$

Since $x_2-v_1\geq \frac{k^{8/7}}{2}$, similar calculations for the fluctuation of $T_{v,(x_2,x_2)}+T^H_{(x_2,x_2),(n+k,n)}$ around $\E(T_{v,(x_2,x_2)})+\E(T^0_{(x_2,x_2),(n+k,n)})$ and union bound over all $x_2\in \llbracket n-\frac{k^{8/7}}{2},n\rrbracket$ give the result.
\end{proof}

\section{Constructing a Candidate for Invariant Measure}
\label{s:inv}
In this section we construct a candidate for invariant measure for TASEP with a slow bond. The construction of the measure is fairly intuitive. Fix a finite interval $[-b,b]$ around the origin. Recall that $T_{n}$ is the last passage time between $(0,0)$ and $(n,n)$ (i.e., $T_{n}$ is the time that $(n+1)$st particle crosses the origin). Now for $n\gg k \gg 1$ consider the average occupation measure of sites in $[-b,b]$ between $T_{n}$ and $T_{n+k}$. Because the geodesics are localized around the diagonal it turns out using the correspondence between TASEP and DLPP that the occupation measures are approximately determined by the weight configuration on a small box around the diagonal between  $(n,n)$ and $(n+k,n+k)$. Moreover using the independence of the configurations of such disjoint boxes and Theorem \ref{meetondiagonal} one can construct such a sequence of occupation measures that are Cauchy. One gets an the candidate measure passing to a limit.

Recall that $\eta_{t}=\eta_{t}^{(r)}$ is the configuration of TASEP with a slow bond started from the step initial condition, i.e., $\eta_t(i)=0$ or $1$ according to whether the site $i$ is vacant or occupied at time $t$. Also for an interval $I\subseteq \Z$ let $\eta_{t}(I)$ denote the configuration restricted to $I$. Our main theorem in this section is the following. 

\begin{theorem}
\label{t:exist}
There exists a measure $Q$ on $\{0,1\}^{\Z}$ with the following property: Fix $b\in \N$ and $\delta>0$. Set $I=[-b,b]$ and let $Q_{I}$ denote the restriction of $Q$ to $I$. Then there exist constants $k_0,c>0$ depending only on $b$, such that for all $k\geq k_0$, 
\[\sup_{n}\P\left(\sup_{A\subseteq \{0,1\}^{I}} \left|\frac{1}{(4+\varepsilon)k}\int_{T_n}^{T_{n+k}}\ind(\eta_t(I)\in A )dt-Q_{I}(A)\right|>\delta\right)<e^{-c\delta k^{1/13}}.\]
\end{theorem}


The main step of the proof of Theorem \ref{t:exist} is to establish that the sequence of average occupation measures as in the statement of the Theorem is almost surely Cauchy. To this end we have the following proposition. 

\begin{proposition}
\label{Cauchyproperty}
Fix $b\in \N$ and $\delta>0$, set $I=[-b,b]$ and fix $A\subseteq \{0,1\}^{I}$. Then there exist absolute constants $k_0,c>0$ such that for all $k\geq k_0$,
\[\sup_{n\in \N, \alpha\in [\frac{3}{2},3]}\P\left(\left|\frac{1}{k^\alpha}\int_{T_n}^{T_{n+k^\alpha}}\ind(\eta_t(I)\in A )dt-\frac{1}{k}\int_{T_n}^{T_{n+k}}\ind(\eta_t(I)\in A)dt\right|>\delta\right)<e^{-c\delta k^{1/12}}.\]
\end{proposition}
In the proof, we shall need the following parallelogram. For $s>b$, $s,b\in \N$, let $U_{s,r,b}$ denote a parallelogram with endpoints $(s-b-1,s),(s,s-b-1),(s+r+b,s+r),(s+r,s+r+b)$.

 Recall that for sites $j\in [-b,b]$, the length of $T_{s+j,s}$ gives the time taken by a particle at $-s$ to first visit site $j+1$. Hence, for any fixed $A\subseteq \{0,1\}^I$, $\int _{T_s}^{T_{s+r}}\ind(\eta_t(I)\in A)dt$, the total occupation time at these sites corresponding to the states defined by $A$ between times $T_s$ and $T_{s+r}$ is a function of the \emph{pairwise differences} in the lengths of the geodesics starting from $(0,0)$ and ending in $U_{s,r,b}$. We state this without proof in the next lemma.
 
\begin{lemma}
\label{l:occden}
For $s,r,b\in \N$, $s>b$, and $I=[-b,b]$ and $A\subseteq \{0,1\}^I$, let $\Theta_{s,r,b}=\{T_{i,j}:(i,j)\in U_{s,r,b}\cap \Z^2\}$ denote the set of lengths of all geodesics starting from $(0,0)$ and ending in $U_{s,r,b}$. Then there exists a function $f=f_{r,b,A}:\R_+^{|\Theta_{s,r,b}|} \mapsto \R_+$ such that for any $c\in \R_+, \mathbf{x}\in \R_+^{|\Theta_{s,r,b}|}$, $f(\mathbf{x}+c)=f(\mathbf{x})$ and
\[\frac{1}{r}\int _{T_s}^{T_{s+r}}\ind(\eta_t(I)\in A)dt=f(\Theta).\]
Also the function $f$ does not depend on the location $s$.
\end{lemma}

We apply this lemma to prove Proposition \ref{Cauchyproperty}. Define $\cB(s,r,b)$ as the Box$((s-b-1,s-b-1),(s+r+b,s+r+b))$ of size $r+2b+1$. Clearly $U_{s,r,b}\subseteq \cB(s,r,b)$. The idea here is to break the $k^\alpha$-sized box at $(n,n)$ into $k^{\alpha-1}$-many $k$-sized boxes, leaving sufficient amount of gap between each box, and use a renewal argument and a law of large numbers to get the required result. More formally, we do the following.

\begin{proof}[Proof of Proposition \ref{Cauchyproperty}]
Fix $n\in \N$ and $\alpha\in [\frac{3}{2},3]$ and let $B:= \cB(n,k^\alpha,b)$. Let $r_k$ be the largest integer such that $r_k(k+k^{1/3})+k^{1/3}\leq k^{\alpha}$. Clearly $r_k\sim k^{\alpha-1}$. Define
\[a_i=n+(i+1)k^{1/3}+ik, \mbox{ \ \ for } i=0,1,2,\ldots,r_k.\]
Define the boxes
 \[B_i=\mathcal{B}(a_i,k,b)\mbox{ for } i= 0,1,2,\ldots,r_k,\]
 and parallelograms
 \[U_i=U_{a_i,k,b}\subseteq B_i \mbox{ for } i= 0,1,2,\ldots,r_k.\]
  Let $k^{1/3}>4b$. Then each of these $k$-sized boxes $B_i$ are separated by a distance of at least $k^{1/3}/2$. Let $p_i=a_i-b-1$ and $q_i=a_i+k+b$, so that $(p_i,p_i)$ and $(q_i,q_i)$ are the endpoints of the box $B_i$. Define $q_{-1}=n$. See Figure \ref{f:cauchy}.

\begin{figure}[h] 
\centering
\includegraphics[width=0.5\textwidth]{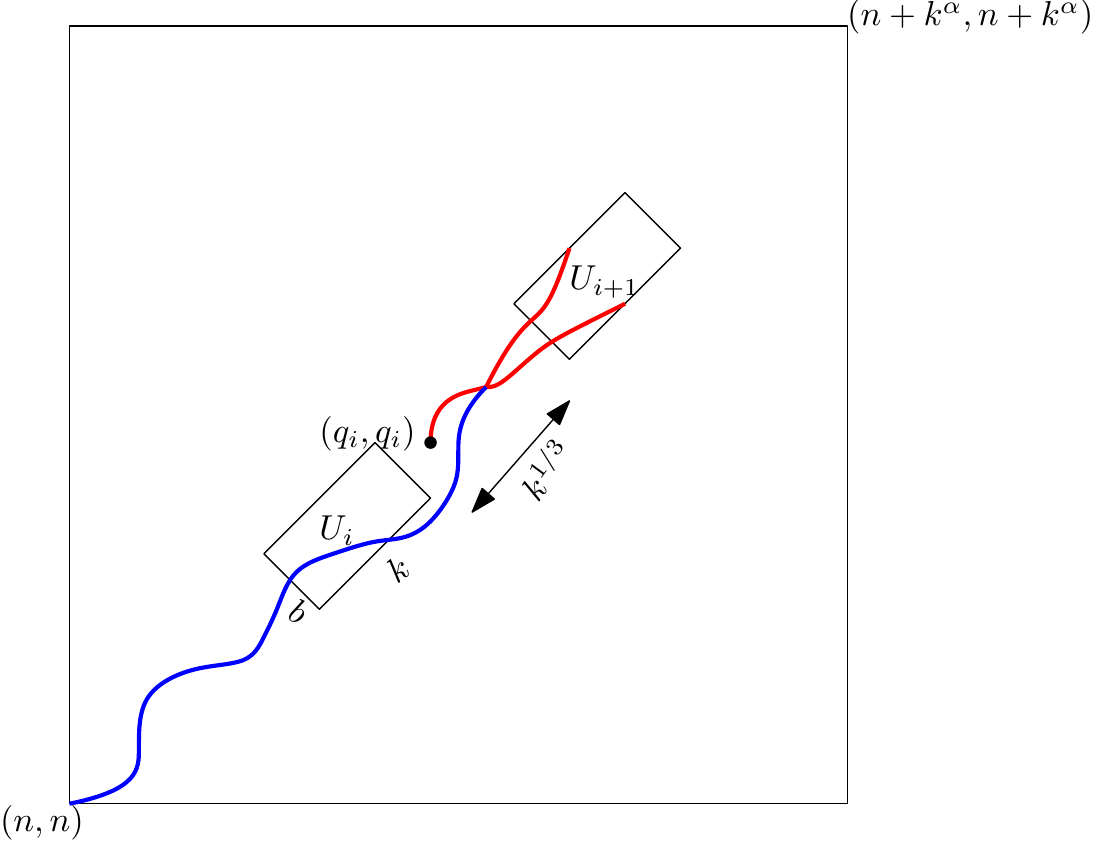} 
\caption{Using the fact that geodesics are localised near the diagonal it follows that on a high probability event the pairwise difference of passage times from the origin to vertices in $U_{i+1}$ is same as the pairwise differences of passage times from the vertex $(q_{i},q_{i})$ to vertices in $U_{i+1}$; the latter collection is an i.i.d.\ sequence.} 
\label{f:cauchy}
\end{figure}


Now fix a particular $i\in \{0,1,2,\ldots,r_k\}$. Let $E_i$ denote the event that all the geodesics from $(0,0)$ to all points in $U_i$ and all the geodesics from $(q_{i-1},q_{i-1})$ to $U_i$ meet the diagonal simultaneously between $\llbracket q_{i-1},p_i \rrbracket$. Then by Corollary \ref{c:meetdiagpar}, 
\[\P(E_i)\geq 1-e^{-ck^{1/12}}.\]
Let $E:=\cap_{i=0}^{r_k} E_i$. Then 
\[\P(E)\geq 1-k^\alpha e^{-ck^{1/12}}\geq 1-e^{-c'k^{1/12}},\]
for all large  enough $k$. Define,
\[\Theta^0_i=\{T_{(0,0),(u,v)}: (u,v)\in U_i\cap \Z^2\},\mbox{ for } i= 0,1,2,\ldots,r_k,\]
\[\Theta^q_i=\{T_{(q_{i-1},q_{i-1}),(u,v)}: (u,v)\in U_i\cap \Z^2\},\mbox{ for } i= 0,1,2,\ldots,r_k.\]
Then on the
event $E$, for all $i\in \llbracket 0, r_k\rrbracket$, $T_{(0,0),(u,v)}-T_{(q_{i-1},q_{i-1}),(u,v)}=T_{(0,0),(u',v')}-T_{(q_{i-1},q_{i-1}),(u',v')}$ for all $(u,v), (u',v')\in U_i$. Using Lemma \ref{l:occden}, there exists $f=f_{k,b,A}$ such that $\frac{1}{k}\int_{T_{a_i}}^{T_{a_i+k}}\ind(\eta_t(I)\in A)dt=f(\Theta_i^0)$. Using the property of translation invariance of $f$, on $E$, we have,
\[\frac{1}{k}\int_{T_{a_i}}^{T_{a_i+k}}\ind(\eta_t(I)\in A)dt=f(\Theta_i^0)=f(\Theta_i^q).\]
Define $Y_i:=f(\Theta_i^q)$ for $i=0,1,2,\ldots,r_k$. Clearly, $Y_i$s are independent and identically distributed. The rest of the argument is standard and uses Chernoff bounds.

First note that
\[\P\left(\left|\frac{1}{k^\alpha}\int_{T_n}^{T_{n+k^\alpha}}\ind(\eta_t(I)\in A)dt-\frac{1}{r_k}\sum_{i=0}^{r_k}\frac{1}{k}\int_{T_{a_i}}^{T_{{a_i}+k}}\ind(\eta_t(I)\in  A)dt\right|\geq \frac{\delta}{8}\right)\leq e^{-ck^{1/12}}. \]
 Indeed, in order to show that $\frac{1}{k^\alpha}\sum_{i=0}^{r_k-1}\int_{T_{a_i+k}}^{T_{{a_{i+1}}}}\ind(\eta_t(I)\in  A)dt$ is small, enough to show that $\frac{1}{k^\alpha}\sum_{i=0}^{r_k-1}\left(T_{{a_{i+1}}}-T_{a_i+k}\right)$ is small. Let $F$ denote the event that for each $i$, the geodesics $\Gamma_{a_{i+1}}$ and $\Gamma_{a_i+k}$ meet together on the diagonal in the interval $\llbracket a_i+k-k^{1/3},a_{i+1} \rrbracket$. Then from Corollary \ref{c:twopathsmeet}, $\P(F)\geq 1-e^{-ck^{1/12}}$. On $F$, 
 \[T_{{a_{i+1}}}-T_{a_i+k}\leq T_{(a_i+k-k^{1/3},a_i+k-k^{1/3}),(a_{i+1},a_{i+1})}.\] 
 As $a_{i+1}-(a_i+k-k^{1/3})=2k^{1/3}$, $T_{(a_i+k-k^{1/3},a_i+k-k^{1/3}),(a_{i+1},a_{i+1})}\overset{d}{=}T_{2k^{1/3}}$. Hence, by union bound,
 \[\P\left(\frac{1}{k^\alpha}\sum_{i=0}^{r_k-1}\left(T_{{a_{i+1}}}-T_{a_i+k}\right)\geq \delta\right)\leq k^{\alpha-1}\P\left(\frac{1}{k}T_{2k^{1/3}}\geq \delta\right).\]
It is easy to see that the right hand side is exponentially small. Similarly one can bound the other terms.
 
  Hence it is enough to get an upper bound to 
\[\P\left(\left|\frac{1}{r_k}\sum_{i=0}^{r_k}\frac{1}{k}\int_{T_{a_i}}^{T_{{a_i}+k}}\ind(\eta_t(I)\in A)dt-\frac{1}{k}\int_{T_n}^{T_{n+k}}\ind(\eta_t(I)\in A)dt\right|>\delta/4\right).\]

Now note that, 
\begin{eqnarray*}
&&\P\left(\left|\frac{1}{r_k}\sum_{i=0}^{r_k}\frac{1}{k}\int_{T_{a_i}}^{T_{{a_i}+k}}\ind(\eta_t(I)\in A)dt-\E\left[\frac{1}{k}\int_{T_n}^{T_{n+k}}\ind(\eta_t(I)\in A)dt\right]\right|>\delta/8\right)\\
&\leq& \P\left(\left\{\left|\frac{1}{r_k}\sum_{i=0}^{r_k}Y_i-\E(Y_0)\right|>\delta/8\right\} \cap E\right)+e^{-ck^{1/12}}\\
&\leq& \P\left(\left|\frac{1}{r_k}\sum_{i=0}^{r_k}Y_i-\E(Y_0)\right|>\delta/8\right)+e^{-ck^{1/12}}, 
\end{eqnarray*}
where $Y_i=f(\Theta_i^q)$, $i=0,1,\ldots, r_k$ are i.i.d.\ having the same distribution as that of the occupation density $\frac{1}{k}\int_{T_{k^{1/3} -b}}^{T_{k^{1/3}-b+k}}\ind(\eta_t(I)\in A)dt$, as discussed earlier. Also, 
\[|Y_i|\overset{d}{=}\left|\frac{1}{k}\int_{T_{k^\beta -b}}^{T_{k^\beta-b+k}}\ind(\eta_t(I)\in A)dt\right|\leq \left|\frac{1}{k}\left(T_{k^\beta-b+k}-T_{k^\beta -b}\right)\right|\leq \frac{1}{k}T_{k+k^\beta}\leq \frac{2T_{k+k^\beta}}{k+k^\beta},\]
and as $\frac{T_{k+k^\beta}}{k+k^\beta}$ is a subexponential random variable (we can take the same parameters for this subexponential random variable for all values of $k$, as when $m$ increases one gets better tail bounds for $\frac{T_m}{m}$), one gets
\[\P\left(\left|\frac{1}{r_k}\sum_{i=0}^{r_k}Y_i-\E(Y_0)\right|>\delta/8\right)\leq e^{-c\delta r_k}\leq e^{-c\delta\sqrt{k}/2}\]
for all $\alpha \in [\frac{3}{2},3]$.
Hence, the only thing left to bound is
\[\P\left(\left|\frac{1}{k}\int_{T_n}^{T_{n+k}}\ind(\eta_t(I)\in A)dt -\E\left[\frac{1}{k}\int_{T_n}^{T_{n+k}}\ind(\eta_t(I)\in A)dt\right]\right|>\delta/8\right).\]

To this end, we follow the exact same procedure as we just did. We consider the Box$((n-b-1,n-b-1),(n+k+b,n+k+b))$ and break it into a number of boxes of size $k^{1/2}$ each, leaving a gap of size $k^{1/3}$ between any two boxes. Define the event $E'$ parallel to the event $E$, such that on $E'$, $\frac{1}{k}\int_{T_n}^{T_{n+k}}\ind(\eta_t(I)\in A)dt$ can be written as a sum of independent and identically distributed random variables. Using exactly similar arguments, we have
\[\P\left(\left|\frac{1}{k}\int_{T_n}^{T_{n+k}}\ind(\eta_t(I)\in A)dt -\E\left[\frac{1}{k}\int_{T_n}^{T_{n+k}}\ind(\eta_t(I)\in A)dt\right]\right|>\delta/8\right)\leq e^{-c\delta k^{1/12}}.\]
This completes the proof.
\end{proof}

Now we are in a position to prove Theorem \ref{t:exist}.
\begin{proof}[Proof of Theorem \ref{t:exist}]
Fix $A\subseteq \{0,1\}^I$ and $n \in \N$. Using $\alpha=2$ in Proposition \ref{Cauchyproperty}, we have,
\[\P\left(\left|\frac{1}{k^2}\int_{T_n}^{T_{n+k^2}}\ind(\eta_t(I)\in A )dt-\frac{1}{k}\int_{T_n}^{T_{n+k}}\ind(\eta_t(I)\in A)dt\right|>\delta\right)<e^{-c\delta k^{1/12}}.\]
Choosing $k=2^{2^m}$ and $\delta=1/2^m$, for all $m$ sufficiently large, one has,
\begin{equation}\label{probbound}
\P\left(\left|\frac{1}{2^{2^{m+1}}}\int_{T_n}^{T_{n+2^{2^{m+1}}}}\ind(\eta_t(I)\in A )dt-\frac{1}{2^{2^m}}\int_{T_n}^{T_{n+2^{2^m}}}\ind(\eta_t(I)\in A)dt\right|>\frac{1}{2^m}\right)<e^{-\frac{c}{2^m} 2^{2^m/12}}.
\end{equation}
As the probabilities are summable in $m$, the sequence $\{\frac{1}{2^{2^m}}\int_{T_n}^{T_{n+2^{2^m}}}\ind(\eta_t(I)\in A)dt\}_m$ is Cauchy almost surely, and hence converges almost surely, the limiting random variable is degenerate by Kolmogorov zero-one law. As $\frac{T_k}{k}\rightarrow (4+\varepsilon)$ a.s.\ as $k\rightarrow \infty$, it is easy to see that 
\[\frac{1}{(4+\varepsilon)2^{2^m}}\int_{T_n}^{T_{n+2^{2^m}}}\ind(\eta_t(I)\in A)dt\overset{\mbox{a.s.}}{\rightarrow} Q_{I}(A),\]
for some probability measure $Q_{I}$ on $\{0,1\}^I$. (It is easy to see that the limit does not depend on $n$). It is not hard to see that $Q_I~$s form a consistent system of probability measures, and hence define a unique probability measure $Q$ on $\{0,1\}^\Z$ such that $Q_I$ is the projection of $Q$ on $I$.

Also, for any fixed $m$ large enough, by summing up the probabilities in \eqref{probbound}, for all $r>m$,
\begin{eqnarray}\label{limitforsubseq}
&&\P\left(\left|\frac{1}{2^{2^m}}\int_{T_n}^{T_{n+2^{2^m}}}\ind(\eta_t(I)\in A)dt-(4+\varepsilon)Q_{I}(A)\right|>\delta\right)\\
&\leq&\P\left(\sup_{r>m}\left|\frac{1}{2^{2^m}}\int_{T_n}^{T_{n+2^{2^m}}}\ind(\eta_t(I)\in A)dt-\frac{1}{2^{2^{r}}}\int_{T_n}^{T_{n+2^{2^{r}}}}\ind(\eta_t(I)\in A )dt\right|>\delta\right)\nonumber \\
&\leq&\sum_{r=m}^\infty \P\left(\left|\frac{1}{2^{2^r}}\int_{T_n}^{T_{n+2^{2^r}}}\ind(\eta_t(I)\in A)dt-\frac{1}{2^{2^{r+1}}}\int_{T_n}^{T_{n+2^{2^{r+1}}}}\ind(\eta_t(I)\in A )dt\right|>\frac{\delta}{2^r}\right) \nonumber\\
&\leq&\sum_{r=m}^\infty e^{-\frac{c\delta}{2^r} 2^{2^r/12}}
\leq 2e^{-c\delta 2^{2^m/13}}\nonumber.
\end{eqnarray}
Now, for any $2^{2^m}\leq k\leq 2^{2^{m+1}}$, there exists some $\alpha\in [\frac{3}{2},3]$, such that $k^\alpha=2^{2^{m+1}}$ or $k^\alpha=2^{2^{m+2}}$. (If $k^\alpha=2^{2^{m+1}}$ for some $\alpha\in[\frac{3}{2},2]$ then we are done, else $k^\alpha=2^{2^{m+1}}$ for some $\alpha\in(1,\frac{3}{2})$, and then $k^{2\alpha}=2^{2^{m+2}}$ where $2\alpha\in(2,3)$). Thus, combining the bounds in Proposition \ref{Cauchyproperty} and that in \eqref{limitforsubseq}, we get,
\begin{equation}
\label{e:limA}
\P\left(\left|\frac{1}{(4+\varepsilon)k}\int_{T_n}^{T_{n+k}}\ind(\eta_t(I)\in A )dt-Q_{I}(A)\right|>\delta\right)<e^{-c\delta k^{1/13}}.
\end{equation}
As this holds for all $A\subseteq \{0,1\}^I$, and $b\in \N$ is fixed, using \eqref{e:limA} for all $2^{2b+1}$ subsets $A$, and union bound,
\[\P\left(\sup_{A\subseteq \{0,1\}^{I}}\left|\frac{1}{(4+\varepsilon)k}\int_{T_n}^{T_{n+k}}\ind(\eta_t(I)\in A )dt-Q_{I}(A)\right|>\delta\right)<e^{-c\delta k^{1/13}}.\]
As this holds for all $n\in \N$, the result follows. 
\end{proof}

\section{Convergence to Invariant Measure}\label{s:conv}
In this section we shall establish that started from the step initial condition TASEP with a slow bond at the origin converges weakly to the measure $Q$. It suffices to prove the following theorem for convergence of finite dimensional distributions.

\begin{theorem}
\label{t:limit}
For any fixed $b\in \N, I=[-b,b]$ and $A\subseteq \{0,1\}^{I}$, 
\[\P(\eta_t(I)\in A)\rightarrow Q_{I}(A)\]
as $t\to \infty$.
\end{theorem}

The idea of the proof goes as follows. First we observe that the configuration of $\eta_{T_n+s}(I)$ at time $T_n+s$ does not depend on the exact value of the passage time $T_n$, but the amount of overshoot of the different passage times from $T_n$, and is thus roughly independent of the passage times near the origin. Also conditioning on all the exponential random variables except at a number of sufficiently spaced vertices on the diagonal near the origin, one can argue that the effects of these vertex weights on $T_n$ are roughly independent. Hence, a local limit theorem suggests that the conditional distribution of $T_n$ is close to Gaussian. Owing to the flatness of the Gaussian density, one can approximate $t-T_n$ by the uniform distribution, and thus reduce $\eta_{t}(I)=\eta_{T_n+(t-T_n)}(I)$ to the average occupation measure over suitable intervals. From here one can resort to Theorem \ref{t:exist} to get the convergence to $Q_I$.


For the proof of Theorem \ref{t:limit} we shall need a few lemmas. The following lemma is basic and follows easily from Theorem \ref{t:exist}.

\begin{lemma}\label{l:exprand}
Fix $b\in \N, I=[-b,b]$ and $A\subseteq \{0,1\}^{I}$ and $n\in \N$ and any $\delta, \alpha>0$. Then, for any random variable $U$ such that $\P(U\in [\frac{n^2}{3},2n^2])\geq 1-\delta$, and $s\geq \alpha n^{1/12}$,
\[\E\left|\frac{1}{s}\int_{T_n+U}^{T_n+U+s}\ind(\eta_t(I)\in A )dt-Q_I(A)\right|\leq 2\delta +e^{-c\delta s^{1/30}}.\]
\end{lemma}

\begin{proof}
Observe that from Theorem \ref{t:exist} with $A=\{0,1\}^I$, it immediately follows,
\[\sup_{n}\P\left( \left|\frac{T_{n+k}-T_n}{(4+\varepsilon)k}-1\right|>\delta\right)<e^{-c\delta k^{1/13}}.\]
This, together with the statement of Theorem \ref{t:exist} gives
\[\P\left(\left|\frac{1}{T_{n+k}-T_n}\int_{T_n}^{T_{n+k}}\ind(\eta_t(I)\in A )dt-Q_{I}(A)\right|>\delta\right)<e^{-c\delta k^{1/13}}.\]
Taking (polynomial in $n$ number of) union bounds over $m,\l$ such that $T_m\leq T_n+U<T_{m+1}$ and $T_{m+\l}\leq T_n+U+s<T_{m+\l+1}$, along with the bounds for $T_m/m$, this implies
\[\P\left(\left\{U\in [\frac{n^2}{3},2n^2]\right\}\cap\left\{\left|\frac{1}{s}\int_{T_n+U}^{T_n+U+s}\ind(\eta_t(I)\in A )dt-Q_I(A)\right|>\delta\right\}\right)<e^{-c\delta s^{1/30}}.\]

Hence,
\begin{eqnarray*}
\E\left|\frac{1}{s}\int_{T_n+U}^{T_n+U+s}\ind(\eta_t(I)\in A )dt-Q_I(A)\right|&\leq& \delta + \P\left(\left|\frac{1}{s}\int_{T_n+U}^{T_n+U+s}\ind(\eta_t(I)\in A )dt-Q_I(A)\right|>\delta\right)\\
&\leq& 2\delta +e^{-c\delta s^{1/30}}.
\end{eqnarray*}
\end{proof}

For the remainder of this section, we shall need the following notations. Fix $n\in \N$ and define 
\[a_i=in^{1/6}, \mbox{ for } i=0,1,2,\ldots,n^{1/6}+1.\]
Let 
$\mathcal{F}_n=\sigma\{\xi_{i,j}:(i,j)\notin \{(a_m,a_m):m=1,2,\ldots
,n^{1/6}\} \}$ denote the $\sigma$-field generated by all the  vertex weights except at the locations $(a_i,a_i)$. Also let \[\mathcal{G}_n=\sigma\{\xi_{i,j}:i\in [0,2n^{1/3}],j\in [0,2n^{1/3}]\}.\]
 For $n\in \N$, define
\[G'_{n}=T_{n,n}-\E(T_{n,n}|\mathcal{F}_n).\]

With these notations, we are in a position to state the next lemma which is required to prove Theorem \ref{t:limit}. 

\begin{lemma}
\label{l:Gsum}
In the above set up, there exists a $\mathcal{G}_n$-measurable random variable $G_n$  such that $\P(G'_{n}=G_n)\geq 1-e^{-cn^{1/24}}$. Moreover $G_n$ is a sum of $n^{1/6}$ many i.i.d.\ random variables with a non lattice distribution and having mean $0$ and variance $\tau^2\in [a,b]$ for some absolute positive constants $a,b$.
\end{lemma}

The argument is standard and almost mimics that in the proof of Proposition \ref{Cauchyproperty}.

\begin{proof}
For any region $B\in \R^2$, let $\Gamma(B):=\Gamma\cap B$ denote the part of $\Gamma$ inside the region $B$, and its weight as $T(B)$. Consider the points \[p_i:=\frac{a_{i-1}+a_{i}}{2}, \mbox{ for } i=1,2,\ldots,n^{1/6}+1.\]
 Further define 
\[m_i=\frac{p_i+a_i}{2}, \mbox{ for } i=1,2,\ldots, n^{1/6}+1, \mbox{ and } n_i=\frac{a_i+p_{i+1}}{2}, \mbox{ for } i=1,2,\ldots, n^{1/6}.\]

\begin{figure}[h] 
\centering
\includegraphics[width=0.2\textwidth]{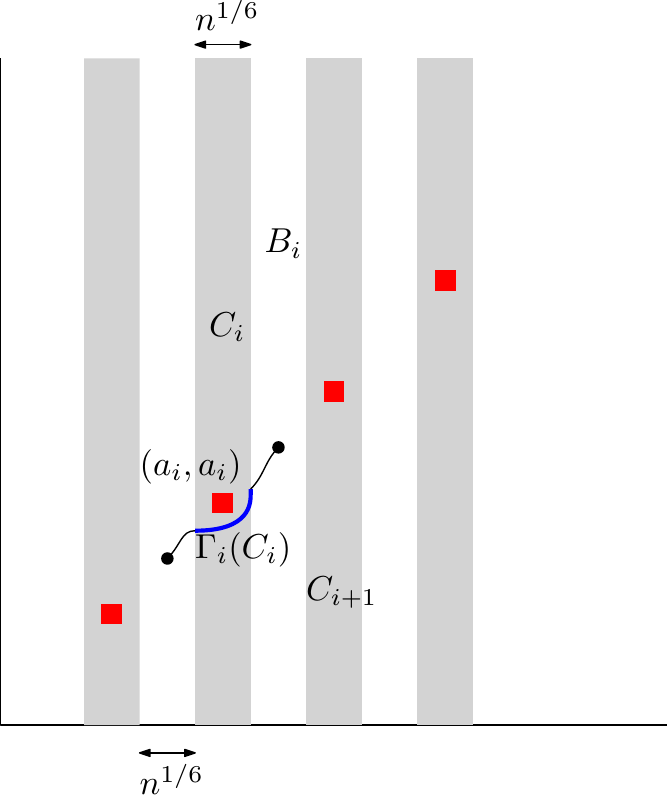} 
\caption{The vertices marked in red are the ones not revealed. $T_{i}(C_i)$ are i.i.d.\ random variables whose sum we use to capture the fluctuation in $T_{n}$ after all the remaining vertices are revealed.} 
\label{f:lclt}
\end{figure}
 
 The geodesic $\Gamma_{(a_i+1,a_i+1),(a_{i+1}-1,a_{i+1}-1)}$ from $(a_i+1,a_i+1)$ to $(a_{i+1}-1,a_{i+1}-1)$ and the geodesic $\Gamma_{n,n}$ meet together on the diagonal between $a_i$ and $n_i$ and again between $m_{i+1}$ and $a_{i+1}$, and hence coincides between $n_i$ and $m_{i+1}$, with probability atleast $1-e^{-cn^{1/24}}$ by Corollary \ref{c:twopathsmeet}. Let \[B_i=\{(x,y):n_i\leq x\leq m_{i+1}\} \mbox{ for } i=1,2,\ldots,n^{1/6},\]
 \[D=\{(x,y):x\geq 2n^{1/3}\},\]
 \[C_i=\{(x,y):m_i\leq x\leq n_i\} \mbox{ for } i=1,2,\ldots,n^{1/6}.\]
 Note that, of these, only the sets $C_i$~s contain the unrevealed locations $(a_i,a_i)$. Since for all $i$, the geodesics $\Gamma_{(a_i+1,a_i+1),(a_{i+1}-1,a_{i+1}-1)}$ are $\mathcal{F}_n$ measurable, hence 
 \[\P\Big(T_{n,n}(B_i)-\E(T_{n,n}(B_i))|\mathcal{F}_n)=0  \mbox{ for all } i\Big) \geq 1-n^{1/6}e^{-cn^{1/24}}.\] 
  A similar argument shows that $T_{n,n}(D)-\E(T_{n,n}(D)|\mathcal{F}_n)=0$ with probability atleast $1-e^{-cn^{1/24}}$. 
  
  Let $\Gamma_i:=\Gamma_{(p_i,p_i),(p_{i+1},p_{i+1})}$ and $T_i(C_i)$ denote the weight of $\Gamma_i\cap C_i$. Then, repeating similar calculations, $\Gamma_i$ and $\Gamma_{n,n}$ coincide inside $C_i$ for all $i=1,2,\ldots,n^{1/6}$ with probability atleast $1-n^{1/6}e^{-cn^{1/24}}$. Define
  \begin{equation}
  \label{defofGn}
  G_n=\sum_i^{n^{1/6}}(T_i(C_i)-\E(T_i(C_i)|\mathcal{F}_n)).
  \end{equation}
 Clearly $G_n$ is $\mathcal{G}_n$-measurable and it follows that
 \begin{equation*}
 \P(G'_{n}=G_n)\geq 1-e^{-cn^{1/24}}.
 \end{equation*}
 Observe that $T_i(C_i)-\E(T_i(C_i)|\mathcal{F}_n)$, $i=1,2,\ldots,n^{1/6}$ are i.i.d.\ mean zero random variables with a non lattice distribution. That they have bounded variance follows from Lemma \ref{l:Gvar}. This completes the proof of Lemma \ref{l:Gsum}.
\end{proof}

Now we begin the proof of Theorem \ref{t:limit}.
\begin{proof}[Proof of Theorem \ref{t:limit}] 
Fix $t\geq 0$ sufficiently large. Fix $n=\floor{\sqrt{t}}$. Fix any $\delta>0$. Recall that $G_n$ was defined in \eqref{defofGn}. Let $M$ be a large enough constant such that $-Mn^{1/12}\leq G_n\leq Mn^{1/12}$ with probability at least $1-\delta$.

Define
\[h(g):=\eta_{(T_n+t-\E(T_n|\cF_n)-g)}(I).\]
Then by Lemma \ref{l:Gsum},
\[\P(h(G_n)=\eta_t(I))\geq 1-e^{-cn^{1/24}}.\]

Now observe that if $\E(T_n|\cF_n) \leq \frac{n^2}{2}$, and $-Mn^{1/12}\leq g \leq Mn^{1/12}$, then $\frac{n^2}{3}\leq t-\E(T_n|\cF_n)-g\leq 2n^2$. Also for any $0\leq c\leq T_{2n^2}-T_n$, $\eta_{T_n+c}(I)$ is a function of the differences $T_{x,y}-T_n$ where $(x,y)\in D$ where $D$ is the $2b$ sized strip along the diagonal in $\mbox{Box}((n-b,n-b),(2n^2+b,2n^2+b))$. Let $E$ be the event that there exists $u\in \llbracket n^{1/2},n-b\rrbracket$ such that $(u,u)\in \bigcap_{(x,y) \in D} \Gamma_{\mathbf{0},(x,y)}\bigcap_{(x,y)\in D} \Gamma_{(n^{1/2},n^{1/2}),(x,y)}$. Then union bound and Corollary \ref{c:meetdiagpar} imply that $\P(E)\geq 1-e^{-cn^{1/8}}$. On $E$, $\eta_{T_n+c}(I)$ is a function of the differences $T_{(n^{1/2},n^{1/2}),(x,y)}-T_{(n^{1/2},n^{1/2}),(n,n)}$ which is $(\mathcal{G}_n)^c$ measurable. Then on the event that $\E(T_n|\mathcal{F}_n)\leq \frac{n^2}{2}$, there exists a function $h'$ which is $(\mathcal{G}_n)^c$ measurable, such that for each $g\in [-Mn^{1/12},Mn^{1/12}]$,
\[\P(h(g)=h'(g+\E(T_n|\mathcal{F}_n)))\geq 1-e^{-cn^{1/8}}.\]

Then,
\begin{eqnarray*}
&&\P(\eta_t(I)\in A)\\
&=& \P(h(G_n)\in A)+R_n\\
&=& \E(\P(h(G_n)\in A)|\mathcal{F}_n)+R_n\\
&=& \E\left(\int\P(h(g)\in A|\mathcal{F}_n,G_n=g)f_{G_n|\mathcal{F}_n}(g)dg\right)+R_n \\
&=&\E\left(\int\P(h'(g+\E(T_n|\mathcal{F}_n))\in A|\mathcal{F}_n,G_n=g)f_{G_n|\mathcal{F}_n}(g)dg\right)+R'_n \\
&=&\E\left(\int \P(h'(g+E(T_n|\mathcal{F}_n))\in A|\mathcal{F}_n)f_{G_n|\mathcal{F}_n}(g)dg\right)+R'_n\\
&=&\int \E(\P(h'(g+E(T_n|\mathcal{F}_n))\in A|\mathcal{F}_n))f_{G_n}(g)dg +R'_n\\
&=&\int \P(h(g)\in A)f_{G_n}(g)dg +R''_n,
\end{eqnarray*}
where $|R_n|, |R'_n|,|R''_n|\leq 3e^{-n^{1/24}}$, by interchanging the integral and expectation, and the fact that given $\mathcal{F}_n$, $h'$ and $G_n$ are conditionally independent.

Fix $A\subseteq \{0,1\}^I$. Since by Lemma \ref{l:unifcont}, $\psi(g)=\P(h(g)\in A)$ is uniformly continuous, choose $h>0$ such that $\sup_{|g-g'|\leq h}|\psi(g)-\psi(g')|\leq \delta$. For this $h>0$, applying local central limit theorem to $G_n$, due to Lemma \ref{l:Gsum}, we have,
\[\P(a\leq G_n\leq a+h)=h\phi_n(a)+o(1/n^{1/12}),\]
where the error term $o(1/n^{1/12})$ is uniform in $a$, and $\phi_n$ denotes the density of $N(0,\tau^2 n^{1/6})$ distribution, where $\tau$ is bounded. Then,
\[\left|\int \P(h(g)\in A)f_{G_n}(g)dg-\int \P(h(g)\in A)\phi_n(g)dg\right|\leq 4\delta +o(1).\]

Now,
\[\int \P(h(g)\in A)\phi_n(g)dg=\int\P(\eta_{T_n+U-g}(I)\in A)\phi_n(g)dg=\E\left(\int\ind(\eta_{T_n+U-g}(I)\in A)\phi_n(g)dg \right), \]
where $U=t-\E(T_n|\cF_n)\in [\frac{n^2}{2},n^2]$  with probability at least $1-\delta$. Now, if $\phi$ denotes the density of $Z\sim N(0,\tau^2)$, then get $R$ large such that $\P(|Z|\geq R)\leq \delta$. Also let $\beta$ be the modulus of continuity of the Gaussian density corresponding to this $\delta$. Divide $[-R,R]$ into points $a_1=-R,a_2,\ldots,a_r=M$ such that $|a_i-a_{i+1}|=\beta$ ( so that $\sup_{x\in [a_i,a_{i+1}]}|\phi(x)-\phi(a_i)|\leq \delta$). Then if $b_i:=n^{1/12}a_i$, then $b_{i+1}-b_i=n^{1/12}\beta$ and,
\[\sup_{x\in [b_i,b_{i+1}]}|\phi_n(x)-\phi_n(b_i)|\leq \frac{\delta}{n^{1/12}}.\]
 
Hence, using Lemma \ref{l:exprand}, one gets,
\begin{eqnarray*}
&&\left|\E\left(\int\ind(\eta_{T_n+U-g}(I)\in A)\phi_n(g)dg \right)-Q_I(A)\right|\\
&\leq&\left|\E\left(\sum_{i=1}^{r-1}\int_{b_i}^{b_{i+1}}
\ind(\eta_{T_n+U-g}(I)\in A)\phi_n(g)dg \right)-Q_I(A)\right|+\delta\\
&\leq&\left|\E\left(\sum_{i=1}^{r-1}\int_{b_i}^{b_{i+1}}
\ind(\eta_{T_n+U-g}(I)\in A)\phi_n(b_i)dg \right)-Q_I(A)\right|+2R\delta+\delta\\
&=&\left|\sum_{i=1}^{r-1}\phi_n(b_i)(b_{i+1}-b_i)\E\left(\frac{1}{b_{i+1}-b_i}\int_{T_n+U-b_{i+1}}^{T_n+U-b_{i}}
\ind(\eta_s(I)\in A)ds \right)-Q_I(A)\right|+2R\delta+\delta\\
&\leq& \sum_{i=1}^{r-1}\phi_n(b_i)(b_{i+1}-b_i)\left(2\delta+e^{-c\delta(b_{i+1}-b_i)^{1/30}}\right)+Q_I(A)\left|\sum_{i=1}^{r-1}\phi_n(b_i)(b_{i+1}-b_i)-1\right|+2R\delta+\delta\\
&=&\sum_{i=1}^{r-1}\phi_n(b_i)(b_{i+1}-b_i)\left(2\delta+e^{-c\delta(n^{1/12}\beta)^{1/30}}\right)+4R\delta+2\delta\\
&\leq& \left(2\delta+e^{-c\delta(n^{1/12}\beta)^{1/30}}\right)+ 2R\delta \left(2\delta+e^{-c\delta(n^{1/12}\beta)^{1/30}}\right)+4R\delta+2\delta.
\end{eqnarray*}
Now let $n\rightarrow \infty$, and then $\delta\rightarrow 0$, and note that $R\delta\rightarrow 0$ as $R\sim \sqrt{\log(1/\delta)}<<\frac{1}{\sqrt{\delta}}$ as $\delta\rightarrow 0$.
\end{proof}

\section{Coupling TASEP with a Slow Bond with a Stationary TASEP}
\label{s:couple}
We complete the proof of Theorem \ref{t:mainst} in this section. Recall $\rho$ from Remark \ref{r:rho}. Now fix $L\in N$ and set $I=[0,L]$. For $k\in \N$, consider  the interval $k+I$. We shall define a coupling between the stationary TASEP with density $\rho$ (i.e., with product $\mbox{Ber}(\rho)$ particle configuration) and the TASEP with a slow bond started from the step initial condition. We shall show that under this coupling for all $k$ sufficiently large, the asymptotic occupation measure for $k+I$ in the slow bond model is with probability close to one equal to the occupation measure of $I$ in the stationary TASEP with density $\rho$. This implies the total variation distance between the two measures is small. By stationarity one of them is close to product $\mbox{Ber}(\rho)$, and hence the other must be so too. This will establish that the limiting stationary measure $\nu_{*}$ of the slow bond process is asymptotically equivalent to $\nu_{\rho}$ at $\infty$. By an identical argument one can establish asymptotic equivalence to $\nu_{1-\rho}$ at $-\infty$. The crux of the argument will be to show that the coupling works with large probability and to show that we first need to consider the last passage percolation formulation of TASEP with an arbitrary initial condition, in particular a stationary one. 

\subsection{Last Passage Percolation and Stationary TASEP}
\label{s:couplestat}
The correspondence we described between TASEP started from step initial condition extends to arbitrary initial condition as follows. Let $\eta\in \{0,1\}^{\Z}$ be an arbitrary particle configuration. Let $S_{\eta}$ be a bi-infinite connected subset of $\Z^2$ (an injective image of $\Z$) defined recursively. Define $F:\Z\to \Z^2$ as follows: set $F(0)=(0,0)$. For $i>0$, set $F(i)=F(i-1)+ (0,-1)$ if $\eta(i)=1$ and set $F(i)=F(i-1)+(1,0)$ otherwise. For $i<0$, set $F(i)=F(i+1)+(0,1)$ if $\eta(i+1)=1$ and $F(i)=F(i+1)+(-1,0)$ otherwise. Let $S_{\eta}=F(\Z)$. Clearly $S_{\eta}$ is a connected subset of $\Z^2$ that divides $\Z^2$ into two connected components (see Figure \ref{f:stationary}) one of which (let us call it $R_{\eta}$) contains the whole positive quadrant. Also set $R=R_{\eta}\cup S_{\eta}$. Fix $\ell\in \Z$; let $a\in \Z$ be the smallest number such that $(a+\ell, a)\in R$. Consider the following coupling between TASEP with initial condition $\eta$ and Last Passage Percolation with i.i.d.\ exponential weights by setting $\xi_{a+j+\ell, a+j}$ (the edge weight at vertex $(a+j+\ell, a+j)$) to be equal to the waiting time for the $(j+1)$-th jump at site $\ell$. The following standard result gives the correspondence between last passage times and jump times under this coupling. We omit the proof, see e.g.\ \cite{BFP10}.

\begin{figure}[h] 
\centering
\includegraphics[width=0.3\textwidth]{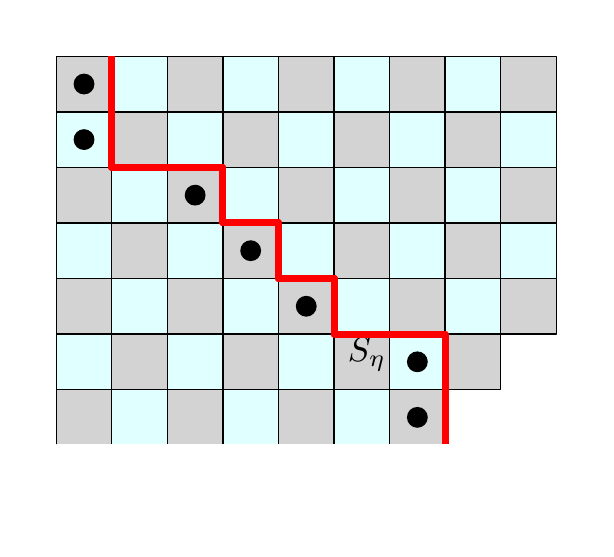} 
\caption{Correspondence between Last Passage Percolation and TASEP with general initial condition. The red line in the figure depicts a part of the boundary $S_{\eta}$ between $R_{\eta}$ and $\Z^2\setminus R_{\eta}$ for a part of the configuration $\eta$ given by 110101010011. Jump times in TASEP started with initial condition $\eta$ corresponds to last passage times from $S_{\eta}$ to various vertices of $\Z^2$.} 
\label{f:stationary}
\end{figure}

\begin{proposition}
\label{p:couplegeneral}
Let $\eta\in \{0,1\}^{\Z}$ be fixed and consider the coupling described above. Fix $\ell\in \Z$ and let $v_{n}=(n+\ell, n)$. Then the time taken for the $n$-th particle to left of the origin to jump through site $\ell$ is equal to
$$\sup_{v\in S_{\eta}} T_{v,v_{n}}$$
where $T_{v,v_{n}}$ denotes the usual last passage time between $v$ and $v_{n}$. 
\end{proposition} 

Observe that in case $\eta=\ind_{(-\infty,0]}$, i.e., for step initial condition the set $S_{\eta}$ is just the boundary of the positive quadrant of $\Z^2$ and hence point to line (or general set $S_{\eta}$) passage time reduce to point to origin passage time in that case and hence this result is consistent with the previous correspondence between TASEP and LPP that we quoted. 


Let us now specialise to the stationary $\mbox{Ber}(\rho)$ initial condition, i.e., in $\eta$ each site is occupied with probability $\rho$ independent of the others. Clearly in such a case the gap between two consecutive particles is a geometric random variable with mean $\frac{1}{\rho}-1$. Hence in this case $S_{\eta}$ is a random staircase curve passing through the origin that has horizontal steps of length that are distributed as i.i.d.\ $\mbox{Geom}(\rho)$ and two consecutive horizontal steps (can also be of length $0$) are separated by vertical steps of unit length. By a large deviation estimate on geometric random variables this random line $S_{\eta}$ can be approximated by the deterministic line $\mathbb{L}$ given by
$$ y=-\frac{\rho}{1-\rho}x;$$ 
so for our purposes we can approximate $S_{\eta}$ with $\L$ and consider the corresponding last passage times to $\mathbb{L}$. Before making a precise statement, let us first consider the last passage percolation to the deterministic line $\L$. Fix $\ell\in \Z$ and consider the geodesic from $\mathbb{L}$ to $(n+\ell, n)$ for $n$ large. If $n\gg \ell$, this geodesic is quite close to the geodesic from $\mathbb{L}$ to $(n,n)$. Comparing the first order of the length of the geodesic from $v$ to $(n,n)$ for different $v\in \mathbb{L}$  it is not too hard to see that the first order distance, i.e., $(\sqrt{n-x}+\sqrt{n-y})^2$ for all $(x,y)\in \L$ is maximized at $(nx_0,ny_0)\in \mathbb{L}$ where $(x_0,y_0)$ is given by

\begin{equation}
\label{defofx0}
(x_0,y_0)=\left(-\frac{1-2\rho}{\rho}, \frac{1-2\rho}{1-\rho}\right).
\end{equation}

\begin{figure}[htbp!] 
\centering
\includegraphics[width=0.4\textwidth]{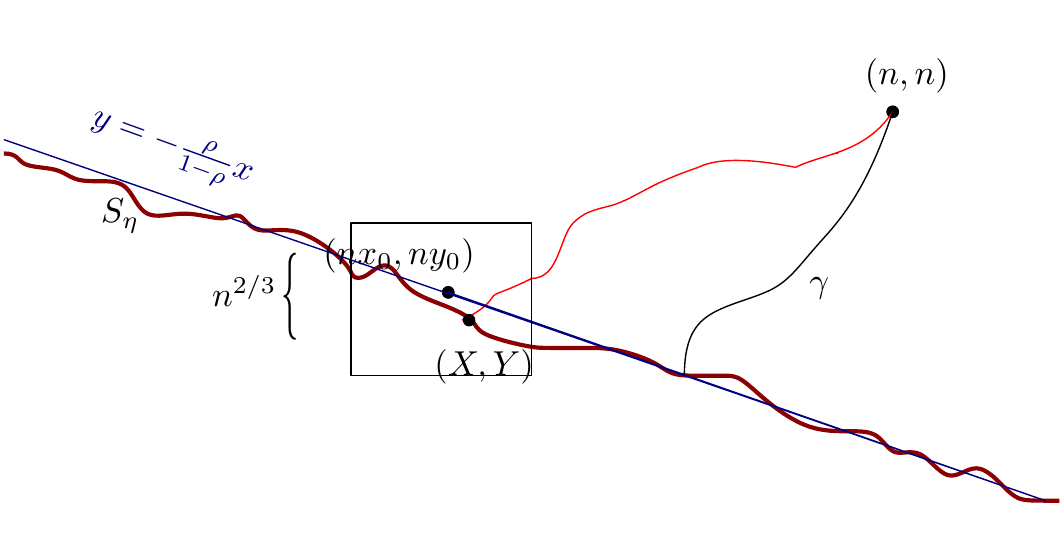} 
\caption{Proposition \ref{p:approximation}: we compare the best path to $S_{\eta}$ from $(n,n)$ to a path $\gamma$ that intersects $S_{\eta}$ far from $(nx_0,ny_0)$. Because of large deviation estimates, at that point $S_{\eta}$ is not too far from the line $y=-\frac{\rho}{1-\rho}x$. By computing expectation we show that that expected length of $\gamma$ is less than that of the best geodesic by an amount larger than the natural fluctuation scale.} 
\label{f:closeS}
\end{figure}

Note that $(nx_0,ny_0)\neq (0,0)$ whenever $\rho\neq \frac{1}{2}$. It follows from this that the point where the geodesic from $(n+\l,n)$ to $\mathbb{L}$ hits $\mathbb{L}$ should be at distance $O(n^{2/3})$ from $(nx_0, ny_0)$. In fact the same remains true for the geodesic from  $(n+\l,n)$ to the random curve $S_{\eta}$. More precisely we have the following proposition. 

\begin{proposition}
\label{p:approximation} Fix $\rho \in (0,1), \rho\neq \frac{1}{2}$. Let $(x_0,y_0)$ be the point on the line $\mathbb{L}$ given by \eqref{defofx0} and let $\ell\in \Z$ be fixed. Let $\Gamma^S_{n+\l,n}$ be the geodesic from $(n+\l,n)$ to the random line $S_{\eta}$. Let $(X_\l,Y_\l)\in S_\eta$ be such that $\Gamma^S_{n+\l,n}=\Gamma^0_{(X_\l,Y_\l),(n+\l,n)}$, where $\Gamma^0_{u,v}$ is the point to point geodesic from $u$ to $v$ in the usual Exponential DLPP. Then given any $\delta>0$ however small, there exists $M=M(\delta,\l)$ such that,
 \[\P(\|(X,Y)-(nx_0,ny_0)\|\geq Mn^{2/3})\leq \delta.\]
\end{proposition}

The proof of this proposition follows from a computation balancing expectation and fluctuation and using Theorem \ref{t:moddevdiscrete} to bound the tails. We shall omit the proof. The argument is by now standard and has been used a number of times in bounding transversal fluctuation for geodesics in various polymer models in KPZ universality class; see e.g.\ \cite[Theorem 11.1]{BSS14}. In the setting of point-to-line last passage percolation this was considered in the very recent preprint \cite{FO17}. Indeed, Lemma $4.3$ of \cite{FO17} shows that the geodesic $(n+\l,n)$ to any point on $\L$ that is more than a distance of $Mn^{2/3}$ from $(nx_0,ny_0)$ is smaller than the geodesic to $(nx_0,ny_0)$ with probability close to one. Also, by suitably applying Chernoff bound, one can show that $|S_\eta(v)-L(v)|=O(n^{1/2})$ for $|v|=O(n)$ with high probability, implying the proposition.

\subsection{Convergence to Product Bernoulli Measure}
We shall complete the proof of Theorem \ref{t:mainst} now. Recall the invariant measure $Q$ constructed in Section \ref{s:inv} and also recall the definition of $\rho<\frac{1}{2}$ from Remark \ref{r:rho}. For an interval $[a,b]$ let $Q_{[a,b]}$ be the projection of $Q$ onto the coordinates in $[a,b]$. It suffices to prove that for each fixed $L\in \N$
\begin{equation}
\label{limitBernoullir}
Q_{[k,k+L]}(\beta_0,\ldots,\beta_L)\overset{k\rightarrow\infty}{\longrightarrow} \rho^{\sum_i \beta_i}(1-\rho)^{L+1-\sum_i\beta_i}, \mbox{\ \ and }
\end{equation}
\begin{equation}
\label{limitBernoullil}
Q_{[-k-L,-k]}(\beta_0,\ldots,\beta_L)\overset{k\rightarrow\infty}{\longrightarrow} \rho^{L+1-\sum_i \beta_i}(1-\rho)^{\sum_i\beta_i}
\end{equation}
for every $\underline{\beta}=(\beta_0,\ldots, \beta_{L})\in \{0,1\}^{L+1}$. This will complete the proof of Theorem \ref{t:mainst} with $Q=\nu_*$. We shall only prove \eqref{limitBernoullir} and the second equation will follow from an identical proof.

We first set up the following notations. Here $k'=ck$ is a constant multiple of $k$ and $n\gg k$. The dependence among the various parameters is summarised below.
\begin{itemize}
\item[1.] $L$ denotes a fixed constant.
\item[2.] $\delta>0$ will denote some predefined quantity however small.
\item[3.] $R$ denotes a sufficiently large constant, to be chosen appropriately later, depending only on $L,\delta$.
\item[4.] $k$ is chosen large enough depending on $L,\delta, R$.
\item[5.] $s$ is chosen large enough depending on $k,L,\delta,R$.
\item[6.] $n\geq s$.
\end{itemize}
 Fix $\rho<\frac{1}{2}$ from Remark \ref{r:rho}. We shall call the LPP model corresponding to the stationary TASEP with product $\mbox{Ber}(\rho)$ configuration as considered in the previous subsection the stationary model, the geodesics from $(m+\l,m)$ to the random line $S_\eta$ as $\Gamma^S_{m+\l,m}$ and its weight as $T_{m+\l,m}^S$. 

Recall that $x_1=n-a(\varepsilon)k$ was defined for the reinforced model in Proposition \ref{reinforcedx1} where $a(\varepsilon)$ was calculated in \eqref{e:defx1} in the Introduction. Also $(x_0,y_0)$ was defined in the stationary model in \eqref{defofx0}. Fix $L,R \in \N$. Define $k'$ such that 
\[k'-k'x_0=k-(R+L).\]
Let $v$ be the vertex where the line joining $(x_1,x_1)$ to $(n+k,n)$ intersect the vertical line $x=n+R+L$, i.e., $v=(v_1,v_2)=(n+R+L,n-\frac{a(\varepsilon)(k-R-L)}{a(\varepsilon)+1})$, and for the reinforced model define
\[B_R:=\mbox{Box}(v,(n+k,n)).\]
Also, for the stationary model, define 
\[B_S:=\mbox{Box}((k'x_0,k'y_0),(k',k')).\]

Since $\rho$ is chosen such that the slopes of the line joining $(x_1,x_1)$ to $(n+k,n)$ match that of the line joining $(k'x_0,k'y_0)$ to $(k',k')$,  hence, the dimensions of the two boxes, $B_S$ in the stationary model, and $B_R$ in the reinforced model are same. 

We consider $2L+R$-sized boxes at the vertices $(k',k')$ and $(n+k,n)$ of the two boxes $B_S$ and $B_R$. To be precise, consider the box Box$((k'-L,k'-L),(k'+R+L,k'+R+L)$ in the stationary model and let $D_S$ be the $L$ sized strip along and above the diagonal of this box, i.e., $D_S$ is the quadrilateral with endpoints $(k'-L,k'),(k',k'),(k'+R,k'+R),(k'+R,k'+R+L)$. Similarly consider the box Box$((n+k-L,n-L),(n+k+R+L,n+R+L))$ in the reinforced model, and let $D_R$ be the $L$ sized strip along and above the diagonal of this box. 

Also slightly enlarge the two boxes $B_S$ and $B_R$, so that $B_S^2:=$ Box$((k'x_0,k'y_0),(k'+R+L,k'+R+L)$, and $B_R^2:=$ Box($v,(n+k+R+L,n+R+L)$). Observe that both the boxes $B_S^2$ and $B_R^2$ have i.i.d.\ $\mbox{Exp}(1)$ random variables at each interior vertex. See Figure \ref{f:couplemain}.

\begin{figure}[h] 
\centering
\begin{tabular}{c}
\includegraphics[width=0.4\textwidth]{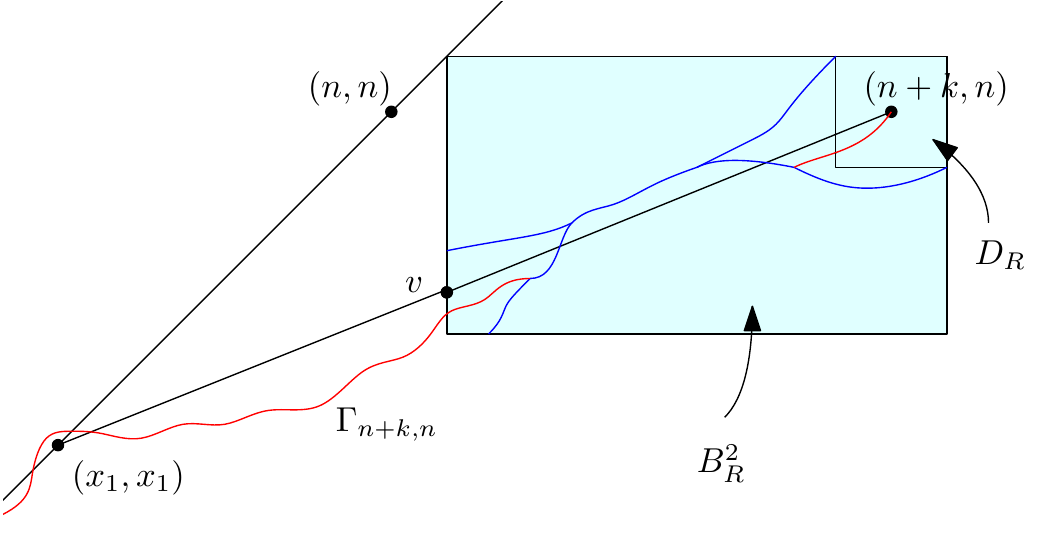} \\
(a)\\
\includegraphics[width=0.4\textwidth]{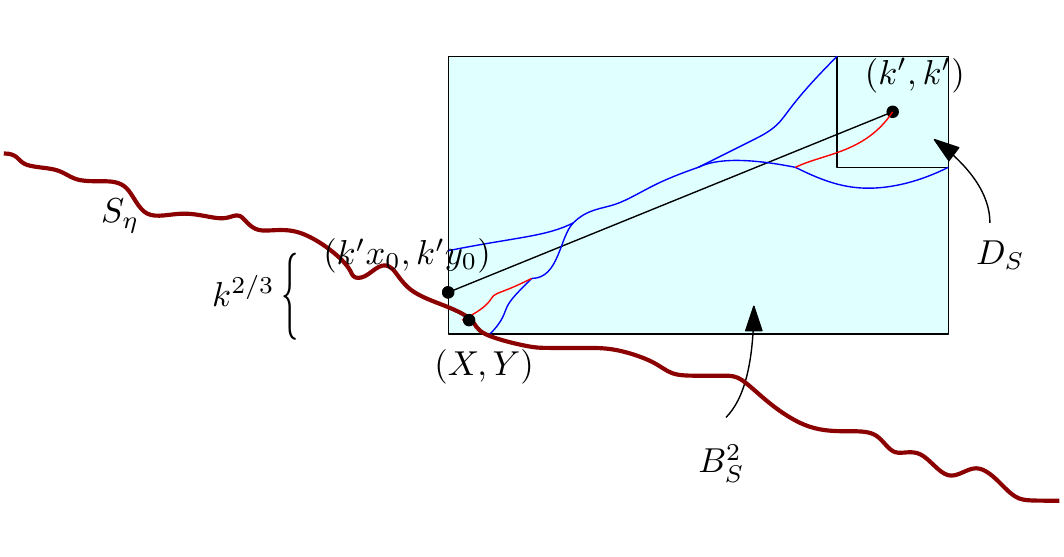}\\(b)
\end{tabular}
\caption{Coupling between TASEP with a slow bond in (a) and Stationary TASEP with density $\rho$ in $(b)$. The point $v$ in (a) corresponds to the point $(k'x_0, k'y_0)$ in (b) and the point $(n+k,n)$ in (a) corresponds to the point $(k',k')$ in (b). We couple the two systems so the the configuration of passage times in the box $B_{R}^2$ in (a) is identical (upto translation) to that in $B_{S}^2$ in (b). We use the fact that the geodesics $\Gamma_{0,u}$ for all $u$ in $D_{R}$ pass through some point close to $v$, and the geodesics $\Gamma_{u}^{S}$ for all $u\in D_{S}$ pass through some point close to $(k'x_0,k'y_0)$ and the coalescence result Theorem \ref{pathsmeet} to argue that the pairwise difference of those geodesics are identical in the two systems with probability close to one if $k$ is large.} 
\label{f:couplemain}
\end{figure}

We have the following basic lemma that relates the expected occupation measures in the reinforced and stationary models. 

\begin{lemma}\label{l:occmeas}
Fix $L,R\in \N$ and consider the above set up. Let $\mathcal{H}_S$ be the event there exists some vertex $u$ such that $u\in \bigcap_{(x,y)\in D_S} \Gamma^S_{x,y}\bigcap_{(x,y)\in D_S} \Gamma^0_{(k'x_0,k'y_0),(x,y)} $. Similarly let $\mathcal{H}_R$ be the event that there exists some vertex $u'$ such that $u'\in \bigcap_{(x,y)\in D_R}\Gamma_{(0,0),(x,y)}\bigcap_{(x,y)\in D_R}\Gamma_{v,(x,y)}$. Then on the event $\mathcal{H}_S\cap \mathcal{H}_R$, under the coupling of all exponential random variables at the corresponding vertices in the two boxes $B_S^2$ and $B_R^2$, for any $A\subseteq \{0,1\}^{L+1}$,
\begin{equation}\label{coupling}
\frac{1}{R}\int_{T_{\mathbf{0},(n+k,n)}}^{T_{\mathbf{0},(n+k+R,n+R)}}\ind(\eta_t[k,k+L]\in A)dt=\frac{1}{R}\int_{T^S_{k',k'}}^{T^S_{k'+R,k'+R}}\ind(\eta_t^S[1,L+1]\in A)dt,
\end{equation}
where $\eta_t$ is the configuration of the TASEP with a slow bond started from the step initial condition and $\eta_t^S$ is the configuration of the stationary TASEP with density $\rho$, and $\eta_t[k,k+L]$ and $\eta_t^S[1,L+1]$ are the configurations restricted to $[k,k+L]$ and $[1,L+1]$.
\end{lemma}
\begin{proof} Let $A_S=\{\Gamma_{x,y}^S:(x,y)\in D_S\}, A_S^0=\{\Gamma^0_{(k'x_0,k'y_0),(x,y)}:(x,y)\in D_S\}, A_R=\{\Gamma_{(0,0),(x,y)}:(x,y)\in D_R\}, A_R^0=\{\Gamma_{v,(x,y):(x,y)\in D_R}\}$. By Lemma \ref{l:occden} it is easy to see that the occupation density 
$\frac{1}{R}\int_{T^S_{k',k'}}^{T^S_{k'+R,k'+R}}\ind(\eta_t^S[1,L+1]\in A)dt=f(A_S)$ for some function $f$ such that $f(\textbf{x}+c)=f(\textbf{x})$. Hence, on $\mathcal{H}_S$, since the differences in the lengths of the maximal paths in the set $A_S$ are the same as the differences in the lengths of the corresponding maximal paths starting from $(k'x_0,k'y_0)$,
\[\frac{1}{R}\int_{T^S_{k',k'}}^{T^S_{k'+R,k'+R}}\ind(\eta_t^S[1,L+1]\in A)dt=f(A_S^0).\]
Similarly, on $\mathcal{H}_R$,
\[\frac{1}{R}\int_{T_{\mathbf{0},(n+k,n)}}^{T_{\mathbf{0},(n+k+R,n+R)}}\ind(\eta_t[k,k+L]\in A)dt=f(A_R^0).\]
Since under the coupling, $f(A_S^0)=f(A_R^0)$, the result follows. 
\end{proof}

The next proposition says that the expected occupation measures in the reinforced and stationary models are close.
\begin{proposition}\label{p:exclose}
Fix  $L,R\in \N$ and $A\subseteq \{0,1\}^{L+1}$ and $\delta>0$. Let $Z_R=\frac{1}{R}\int_{T_{\mathbf{0},(n+k,n)}}^{T_{\mathbf{0},(n+k+R,n+R)}}\ind(\eta_t[k,k+L]\in A)dt$ and $Z_S=\frac{1}{R}\int_{T^S_{k',k'}}^{T^S_{k'+R,k'+R}}\ind(\eta_t^S[1,L+1]\in A)dt$ be the occupation densities in the reinforced and stationary models. Then there exist positive constants $C,c$ depending only on $\delta,L,R$ (and not on $k$) such that
\[|\E Z_R-\E Z_S|\leq \sqrt{\frac{C}{ k^c}+\delta}. \]
\end{proposition}
\begin{proof} First observe that, following similar arguments as in Lemma \ref{uniformintegrable}, the random variables $Z_S$ and $Z_R$, which are bounded by $\frac{T_{k'+R,k'+R}^S-T_{k',k'}^S}{R}$ and $\frac{T_{\mathbf{0},(n+k+R,n+R)}-T_{\mathbf{0},(n+k,n)}}{R}$, are $\mathcal{L}^2$ bounded. Let $C^*$ be the sum of squares of these two $\mathcal{L}^2$ bounds, so that $C^*$ is an absolute positive constant, not depending on $k,R,\delta$.

Recall the events $\mathcal{H}_S, \mathcal{H}_R$ defined in Lemma \ref{l:occmeas}. Let $U_S$ be the line segment joining $(k'x_0+M(k')^{2/3},k'y_0-2M(k')^{2/3})$ and $(k'x_0+M(k')^{2/3},k'y_0+2M(k')^{2/3})$ for some fixed $M=M(\delta)$. Let $E_S$ be the event that the geodesics in $A_S:=\{\Gamma^S_{x,y}:(x,y)\in D_S\}$ and the geodesics in $A_S^0=\{\Gamma^0_{(k'x_0,k'y_0),(x,y)}:(x,y)\in D_S\}$ pass through the line segment $U_S$. Then Proposition \ref{p:approximation} and Theorem \ref{t:carsestimate} together with polymer ordering (Lemma \ref{l:porder}) imply that, for the given $\delta>0$, one can choose $M=M(\delta)$ large enough such that
\[\P(E_S)\geq 1-\frac{\delta}{2C^*}.\]
Let $F_S$ denote the event that the geodesic from $(k'x_0+M(k')^{2/3},k'y_0-2M(k')^{2/3})$ to $(k'+L+R,k'-L)$ and the geodesic from $(k'x_0+M(k')^{2/3},k'y_0+2M(k')^{2/3})$ to $(k'-L,k'+L+R)$ meet together. Then it follows from Theorem \ref{pathsmeet} that
\[\P(F_S)\geq 1-\frac{C'}{k^{c'}},\]
for some $C',c'$ depending on $M$ and hence on $\delta$, but not on $k$.
Due to polymer ordering, 
\[\P(\mathcal{H}_S)\geq \P(E_S\cap F_S)\geq 1-\frac{\delta}{2C^*}-\frac{C'}{k^{c'}},\]

Now we consider the reinforced model. Recall that $v=(v_1,v_2)=(n+R+L,n-\frac{a(\varepsilon)(k-R-L)}{a(\varepsilon)+1})$. Let $\Gamma_1$ be the geodesic from $(x_1+k^{3/5},x_1+k^{3/5})$ to $(n+k-L,n+L+R)$ that avoids the diagonal line segment joining $(0,0)$ to $(n+R+L,n+R+L)$. Also let $\Gamma_2$ be the geodesic from $(x_1-k^{3/5},x_1-k^{3/5})$ to $(n+k+L+R,n-L)$ that avoids the diagonal line segment joining $(0,0)$ to $(n+R+L,n+R+L)$. Then we claim that one can choose $M=M(\delta)$ large enough such that
\begin{equation}\label{e:passur}
\P\left(\{|\Gamma_1(n+R+L)-v_2|\leq Mk^{2/3}\}\cap\{|\Gamma_2(n+R+L)-v_2|\leq Mk^{2/3}\} \right)\geq 1-\frac{\delta}{2C^*}.
\end{equation}
To see this, define $\Gamma_1'$ to be the geodesics from $(x_1+k^{3/5},x_1+k^{3/5})$ to $(n+k-L,n+L+R)$  with the diagonal line segment joining $(0,0)$ to $(n+R+L,n+R+L)$ not reinforced (i.e. corresponding to the usual exponential DLPP). Observe that $\Gamma_1$ can never be above $\Gamma_1'$ and by Theorem \ref{t:carsestimate}, one can choose $M(\delta)$ such that $\P(\Gamma'_1(n+R+L)\leq v_2+Mk^{2/3})\geq 1-\frac{\delta}{4C^*}$. Also, if $\Gamma_2(n+R+L)\leq v_2-Mk^{2/3}$, then for any such $u_2\leq v_2-Mk^{2/3}$, there exists some constant $\alpha$ such that for $u=(n+R+L,u_2)$,
\begin{equation*}
\E(T^0_{(x_1-k^{3/5},x_1-k^{3/5}),u})+\E(T^0_{u,(n+k+L+R,n-L)})\leq \E(T^0_{(x_1-k^{3/5},x_1-k^{3/5}),(n+k+L+R,n-L)})-\alpha M k^{1/3}.
\end{equation*}
Using Proposition $12.2$ of \cite{BSS14} for fluctuations of constrained paths, and moderate deviation estimates of supremum and infimum of geodesic lengths in Proposition $10.1$ and $10.5$ of \cite{BSS14}, and using standard arguments, one gets $\P(\Gamma_2(n)\geq v_2-Mk^{2/3})\geq 1-\frac{\delta}{4C^*}$ by choosing $M$ large.

\begin{figure}[h] 
\centering
\includegraphics[width=0.5\textwidth]{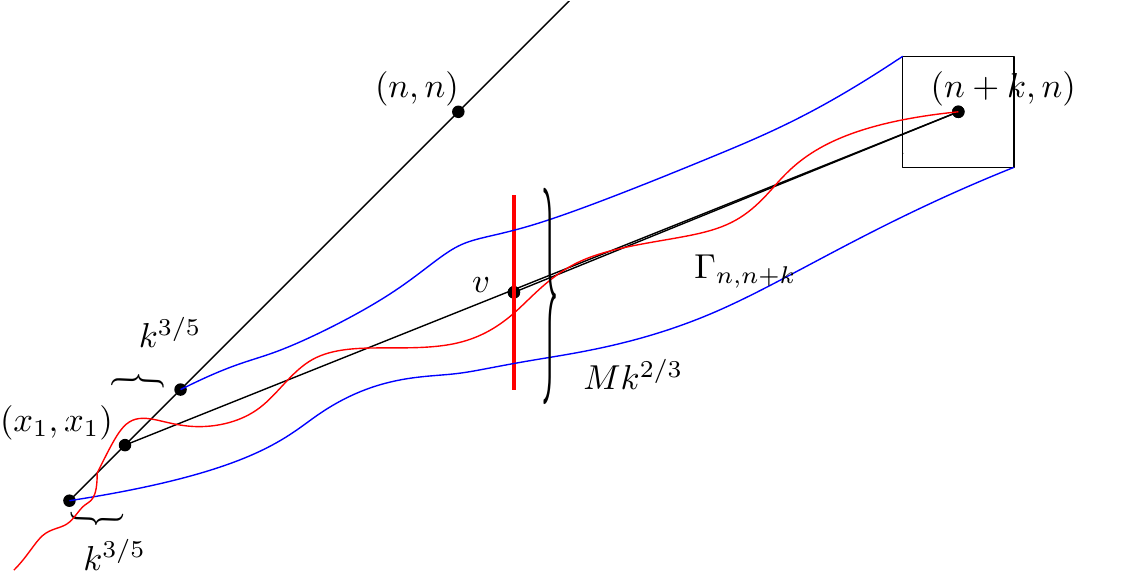} 
\caption{On the high probability event that $\Gamma_{n+k,n}$ leaves the diagonal between $x_1-k^{3/5}$ and $x_1+k^{3/5}$, it is sandwiched between the two geodesic marked in blue in the figure. By local path regularity estimate both of the blue paths are likely to intersect the vertical line through $v_2$ at a distance no more than $Mk^{2/3}$ for some large constant $M$, and hence the same is true for $\Gamma_{n+k,n}$ with probability close to one.} 
\label{f:closeR}
\end{figure}


Let $U_R$ be the line segment joining $(n+R+L,v_2-Mk^{2/3})$ and $(n+R+L,v_2+Mk^{2/3})$. Let $E_R$ be the event that the geodesics in $A_R=\{\Gamma_{(0,0),(x,y)}:(x,y)\in D_R\}$ pass through the line segment $U_R$. Then Proposition \ref{reinforcedx1} and equation \eqref{e:passur} and polymer ordering imply 
\[\P(E_R)\geq 1-\frac{\delta}{2C^*}-C'e^{-ck^{1/20}}.\]
Let $F_R$ denote the event that the geodesic from $(n+R+L,v_2-Mk^{2/3})$ to $(n+k+L+R,n-L)$ and the geodesic from $(n+R+L,v_2+Mk^{2/3})$ to $(n+k-L,n+L+R)$ meet together. Then it follows from Theorem \ref{pathsmeet} that there exists constants $C',c'$ depending on $M$ and hence on $\delta$ such that
\[\P(F_R)\geq 1-\frac{C'}{k^{c'}}.\]
Hence, again by polymer ordering,
\[\P(\mathcal{H}_R)\geq \P(E_R\cap F_R)\geq 1-\frac{C'}{k^{c'}}-\frac{\delta}{2C^*}. \]
 
 Let $Z=Z_R-Z_S$ under the coupling in Lemma \ref{l:occmeas}. Hence, due to Lemma \ref{l:occmeas} and using Cauchy-Schwarz inequality, 
\[|\E Z_R-\E Z_S|\leq \E|Z|=\E |Z|\ind_{(\mathcal{H}_S\cap \mathcal{H}_R)^c}\leq (C^*)^{1/2}\sqrt{\P((\mathcal{H}_S\cap \mathcal{H}_R)^c)}\leq \sqrt{\frac{C}{ k^c}+\delta}.\]
\end{proof}

To finish off the proof we shall show that the occupation measures are indeed close to the measure $Q$ and the product Bernoulli measures respectively. To this end, we have the following lemma.

\begin{lemma}
\label{LHSlimit}
Fix $L,R\in \N$ and $A\subseteq \{0,\}^{L+1}$. Recall that $T_m:=T_{\mathbf{0},(m,m)}$ in the reinforced model and $Z_R:=\frac{1}{R}\int_{T_{\mathbf{0},(n+k,n)}}^{T_{\mathbf{0},(n+k+R,n+R)}}\ind(\eta_t[k,k+L]\in A)dt$. Then for $n\geq s\gg k^2$, $k\geq R$,
\[\left|\E\left(Z_R\right)-\E\left(\frac{1}{sR}\int_{T_{n}}^{T_{n+sR}}\ind(\eta_t[k,k+L]\in A)dt\right) \right|\leq 
\frac{C'k}{sR}+C'e^{-cs^{1/4}}\leq \frac{C'}{kR}+C'e^{-ck^{1/2}},\]
where $C',c$ are constants not depending on $k,n,s,R$.
\end{lemma}

\begin{proof} Let $T^{n,k,a}:=T_{\mathbf{0}, (n+k+a,n+a)}$. Recall the definitions of $D_R$ and $A_R$ from Lemma \ref{coupling}. Using standard arguments, it is easy to see that, for $n\geq s\gg k^2$, the geodesics in $A_R$ and the geodesics starting from $(n-L-s,n-L-s)$ to the corresponding points in $D_R$ meet the diagonal simultaneously with probability atleast $1-e^{-cs^{1/4}}$ by Corollary \ref{c:twopathsmeet}. This also holds true for $n$ replaced by $n+R,n+2R,\ldots, n+(s-1)R$.

Define the random variables 
\begin{eqnarray*}
Y_i&:=&\frac{1}{R}\int_{T^{n,k,(i-1)R}}^{T^{n,k,iR}}\ind(\eta_t[k,k+L]\in A)dt\mbox{,\ \ \ \ }i=1,2,\ldots,s,\\
Z&:=&\frac{1}{R}\int_{T^{s+L,k,0}}^{T^{s+L,k,R}}\ind(\eta_t[k,k+L]\in A)dt.
\end{eqnarray*}
Then, using standard arguments, it follows that there exist random variables $Z_i, i=1,2,\ldots,s$, such that $Z_i\overset{d}{=}Z$ for each $i$, and $\P(Y_i=Z_i \mbox{ for all } i)\geq 1-se^{-cs^{1/4}}$. Hence, using the $\mathcal{L}^2$ boundedness of the random variables and Cauchy-Schwarz inequality, as in the previous proposition, we have, 
\[|\E(Y_i)-\E(Y_1)|\leq |\E(Y_i)-\E(Z)|+|\E(Z)-\E(Y_1)|\leq C'e^{-\frac{c}{2}s^{1/4}}, \mbox{ for each } i=1,2,\ldots, s,\]
where $C'$ is some constant not depending on $n,k,R,s$. Hence,
\[\left|\E\left(\frac{1}{s}\sum_{i=1}^{s} Y_i\right)-\E(Y_1)\right|\leq C'e^{-\frac{c}{2}s^{1/4}}.\]
Now note that, 
\[\frac{1}{s}\sum_{i=1}^{s} Y_i=\frac{1}{sR}\int_{T^{n,k,0}}^{T^{n,k,sR}}\ind(\eta_t[k,k+L]\in A)dt,\]
and $\E\left(\frac{T^{n,k,0}-T_n}{sR}\right)$ and $\E\left(\frac{T^{n,k,sR}-T_{n+sR}}{sR}\right)$ are less than $\frac{C'k}{sR}$ by again using Lemma \ref{uniformintegrable}.
\end{proof}

Finally putting all of these together we get the result. 

\begin{theorem}
\label{t:finalapprox}
Let $\nu_\rho$ denote the product Bernoulli($\rho$) measure. Then, for the given $\delta>0$ and any set $A\subseteq \{0,1\}^{L+1}$,
\begin{equation}
\limsup_{k\rightarrow \infty}|Q_{[k,k+L]}(A)-\nu_\rho(A)|\leq (2\sqrt{\delta}+4\delta)(4+\varepsilon)^{-1}.
\end{equation}
\end{theorem}
As this holds for all $\delta>0$ and all $A\subseteq \{0,1\}^{L+1}$, this completes the proof of Theorem \ref{t:mainst}.
\begin{proof} 
Fix $A\subseteq \{0,1\}^{L+1}$.
Observe that, following similar arguments as in Lemma \ref{uniformintegrable}, the random variables $Z_S$ (which are bounded by $\frac{T_{k'+R,k'+R}^S-T_{k',k'}^S}{R}$), are $\mathcal{L}^2$ bounded, hence uniformly integrable. This, together with stationarity, would imply that, for the given $\delta>0$, we can choose $R=R(\delta)$ large enough not depending on $k'$, such that
\[\left|\E\left(\frac{1}{R}\int_{T^S_{k',k'}}^{T^S_{k'+R,k'+R}}\ind(\eta_t^S[1,L+1]\in A)dt\right)-(4+\varepsilon)\nu_\rho(A)\right|\leq \delta.\]
This is proved in Lemma \ref{l:stapprox} below. Fix such an $R$. Then, by Proposition \ref{p:exclose} and Lemma \ref{LHSlimit}, it follows that for $n\geq s\gg k^2$,
\begin{equation}\label{e:avgmu}
\left|\E\left(\frac{1}{sR}\int_{T_n}^{T_{n+sR}}\ind(\eta_t[k,k+L]\in A)dt\right)-(4+\varepsilon)\nu_\rho(A) \right|\leq \delta+\sqrt{\left(\frac{C}{ k^c}+\delta\right)}+\frac{C'}{kR}+C'e^{-ck^{1/2}},
\end{equation}
where $C,C',c$ depend only on $R,\delta, L$. Choose $k$ large enough so that the right side of \eqref{e:avgmu} is less than $2\sqrt{\delta}+2\delta$. Fix such a $k$.

Applying Theorem \ref{t:exist} and uniform integrability of the random variables, there exist constants $\tilde{C},\tilde{c}$ depending on $k$, such that for every $n\in \N$
\[\left|\E\left(\frac{1}{sR}\int_{T_n}^{T_{n+sR}}\ind(\eta_t[k,k+L]\in A)dt\right)-(4+\varepsilon)Q_{[k,k+L]}(A)\right|\leq \delta +\tilde{C}e^{-\tilde{c}\delta s^{1/13}}.\]
Choose $s$ large so that the right hand side of the above equation is at most $2\delta$.

Combining all this, we get, for any fixed $A\subseteq \{0,1\}^{L+1}$, and for all large $k$ (depending on $\delta$),
\[\left|Q_{[k,k+L]}(A)-\nu_\rho(A)\right|\leq (2\sqrt{\delta}+4\delta)(4+\varepsilon)^{-1}.\]
\end{proof}

\begin{lemma}
\label{l:stapprox}
In the setting of the proof of Theorem \ref{t:finalapprox}, for the given $\delta>0$, there exists $R=R(\delta)$ such that \[\sup_{k'}\left| \E\left(\frac{1}{R}\int_{T^S_{k',k'}}^{T^S_{k'+R,k'+R}}\ind(\eta_t^S[1,L+1]\in A)dt\right)-(4+\varepsilon)\nu_\rho(A)\right|\leq \delta.\]
\end{lemma}

\begin{proof}
The proof is by a standard size biasing argument. Let $\kappa>0$ be fixed sufficiently small depending on $\delta$. Let $\mathcal{A}_{k',R}$ denote the event that
$$\left|\left(\frac{1}{R}\int_{T^S_{k',k'}}^{T^S_{k'+R,k'+R}}\ind(\eta_t^S[1,L+1]\in A)dt\right)-(4+\varepsilon)\nu_\rho(A)\right|\geq \kappa.$$
Clearly it suffices to show that $\sup_{k'}\P(\mathcal{A}_{k',R})\to 0$ as $R\to \infty$. Now, for the process in equilibrium let us denote the law by $\widetilde{\P}$ and the expectation by $\widetilde{\E}$ and let $\Delta$ denote the time difference between two consecutive jumps at the origin. Clearly the distribution of the time difference between the jumps straddling time $0$ is size biased distribution of $\Delta$, and Cauchy-Schwarz inequality then implies 
$$\P(\mathcal{A}_{k',R})\leq  \frac{(\widetilde{\E} \Delta^2)^{1/2}\widetilde{\P}(A_{k',R})^{1/2}}{\widetilde{\E} \Delta}.$$
We know that $\widetilde{\E}\Delta= (4+\varepsilon)$ (cf. Remark \ref{r:rho}) and we have already shown (by Lemma \ref{uniformintegrable}) that $\widetilde{\E} \Delta^2 <\infty$, hence it suffices to prove that $\sup_{k'}\widetilde{\P}(\mathcal{A}_{k',R})\to 0$ as $R\to \infty$. Now observe that $\widetilde{\P}$ measure of $\mathcal{A}_{k',R}$ is independent of $k'$, and hence it suffices to show that
$$ \P\left(\left|\frac{1}{R}\int_{0}^{T_R}\ind(\eta_t^{\tilde{\nu}}[1,L+1]\in A)dt-(4+\varepsilon)\nu_\rho(A)\right|\geq \kappa\right) \to 0$$ as $R\to \infty$ where $\eta^{\tilde{\nu}}$ denotes the process started from the hitting distribution $\tilde{\nu}$ of $\mathcal{B}$ in the stationary chain where $\mathcal{B}$ denotes the set of configurations immediately after a jump at the origin. The result now follows by observing that starting from $\tilde{\nu}$, TASEP converges to $\nu_{\rho}$ weakly and the fact that $\frac{T_{R}}{R}\rightarrow (4+\varepsilon)$ almost surely.
\end{proof}

\subsection{A Sketch of Proof of Theorem \ref{t:otherrho}}\label{s:otherrho}
We end with a sketch of the proof of Theorem \ref{t:otherrho}. As will be clear shortly, the proof is quite similar to the proof of Theorem \ref{t:mainst}, so we shall omit the details. Fix $p<\rho$. We shall show that starting from $\nu_{p}$ initial condition, TASEP with a slow bond converges to a stationary distribution $\nu^*_{p}$, moreover, $\nu^*_{p}$ is asymptotically equivalent to $\nu_{p}$ at $\pm \infty$. A similar argument applies for $p>1-\rho$.

Using the correspondence between TASEP starting from a stationary distribution and last passage percolation described in Subsection \ref{s:couplestat}, it follows that jump times in TASEP started with product $\mbox{Ber}(p)$ initial condition corresponds roughly to last passage times to the line $\mathbb{L}$ with the equation 
$$y=-\frac{p}{1-p}x.$$ Using coalescence of geodesics from points near $(n,n)$ to the line $\mathbb{L}$ it follows as before that average occupation measures over large intervals $T_{n}$ to $T_{n+k}$ converge to a measure $\nu^*_{p}$ on the space of all configurations. However, as the geodesics now will typically not remain pinned to the diagonal, instead of the strong coalescence results of Theorem \ref{meetondiagonal} used earlier, here one has to use Theorem \ref{pathsmeet} for the coalescence of geodesics. To show that the process itself converges to the measure $\nu^*_{p}$ (and hence $\nu^*_{p}$ is stationary), one needs a smoothing argument as in Section \ref{s:conv}. However as the vertices on the diagonal closer to the origin are no longer pivotal, a different argument would be needed. Consider TASEP with a slow bond. By coalescence, the geodesics from points near $(2n,2n)$ to $\mathbb{L}$ are very unlikely to be affected by the first $\frac{n}{4}$ many passage times on the diagonal, in particular one can replace these by i.i.d.\ $\mbox{Exp}(1)$ variables and get a coupling between TASEP with a slow bond; and Stationary TASEP with density $p$ run for time $n$ followed by TASEP with a slow bond, such that the average occupation measure of the former in an interval around time $T_{2n}$ is with high probability identical to that of a slow bond at time  $T_{2n}-\delta n$ for all $\delta\in (0,1)$ as running the stationary TASEP for time $n$ does not change the marginal distribution. Since the occupation measures are close to one another in total variation distance (and each of them are close to $\nu^*_{p}$) the process must converge to the limiting distribution $\nu^*_{p}$. 

It remains to show that $\nu^*_{p}$ is asymptotically equivalent to $\nu_{p}$ at $\pm \infty$. We shall only sketch that $\nu^*_{p}$ is asymptotically equivalent to $\nu_{p}$ at $\infty$, the other part is easier. As in the proof of Theorem \ref{t:mainst}, the basic objects of study are the geodesics to the points $(n+k,n)$ for $n\gg k \gg 1$. The important observation is the following. If $p<\rho$, then the geodesics from $(n+k,n)$ to $\mathbb{L}$ spends only $O(1)$ time on the diagonal, in a deterministic interval of length $O(k^{2/3})$. This can be checked by doing a first order calculation as in Subsection \ref{s:heuristic}, and a variant of Theorem \ref{t:carsestimate}. So the geodesic from $(n+k,n)$ to $\mathbb{L}$ is asymptotically a straight line that has the same slope (asymptotically for $n\gg k\gg 1$) as the geodesic from $(k,k)$ to $\mathbb{L}$ in the unreinforced DLPP. Using this and coalescence one can again couple the occupation measure of stationary TASEP of density $p$ near the origin, to be close in total variation distance to the occupation measure of TASEP with a slow bond at some large time and at sites near the point $k$ for some large $k$. The proof of Theorem \ref{t:otherrho} is then completed as in the proof of Theorem \ref{t:mainst}. We omit the details.

\bibliography{slowbond}
\bibliographystyle{plain}

\section{Appendix A: Proofs of a few technical results} \label{s:techlem}
\subsection{Lemmas used in Section \ref{s:conv}}
\begin{lemma}
\label{l:Gvar}
In the set up of Lemma \ref{l:Gsum}, there exist two absolute positive constants $a,b$, such that $\tau^2\in [a,b]$.
\end{lemma}
The non trivial part in this lemma is to prove the lower bound for $\tau^2$. For this, we follow the proof of Lemma \ref{l:diagpos} and even construct the same events to ensure that $\Gamma_1$ touches $(a_1,a_1)$ on some event whose probability is bounded away from $0$ (not depending on $n$). There is one additional technicality here which is taken care of by the monotonicity of the events.


\begin{proof} Let $X_1$ be the weight of the geodesic from $(m_1,m_1)$ to $(a_1,a_1)$ (excluding the weight of $\xi_{(a_1,a_1)}$), and $X_2$ be the weight of the geodesic from $(a_1,a_1)$ to $(n_1,n_1)$ (excluding the weight of $\xi_{(a_1,a_1)}$), and let $X$ be the weight of the geodesic from $(m_1,m_1)$ to $(n_1,n_1)$ avoiding the point $(a_1,a_1)$. Then clearly,
\[T_1(C_1)= \max\{X_1+X_2+\xi_{(a_1,a_1)}, X\}\leq \max\{X_1+X_2,X\}+\xi_{(a_1,a_1)}.\]
Hence,
\[\max\{X_1+X_2, X\}\leq \E(T_1(C_1)|\mathcal{F}_n)\leq \max\{X_1+X_2, X\}+\E(\xi_{(a_1,a_1)}).\]
Hence,
\[|T_1(C_1)-\E(T_1(C_1)|\mathcal{F}_n)|\leq \xi_{(a_1,a_1)}+\E(\xi_{(a_1,a_1)}).\]
Thus, the upper bound of $\tau^2$ is immediate. Hence we only need to prove the lower bound of $\tau^2$.

Let $C$ be an absolute positive constant to be chosen appropriately later and consider the two boxes of size $C$ whose top left or bottom right vertex is $(a_1,a_1)$. Let $D_1$ denotes the event that the sum of all $C^2-1$ many exponential random variables excluding $\xi_{(a_1,a_1)}$ inside each of these boxes is less than $2C^2$. Let $D_2$ denote the event that $\Gamma_1$ is within a vertical and horizontal distance of $C$ from $(a_1,a_1)$. It is not hard to see that the same argument as in Lemma \ref{i:davoid} and Proposition \ref{deviationfromdiagonal} works even when one point on the diagonal is conditioned to have $0$ weight, and it follows that $\P(D_2|\xi_{(a_1,a_1)}=0)\geq 1-e^{-cC^{1/2}}$, where $c$ is some absolute positive constant. Let $\Pi_B$ denote the configuration restricted to the set $B$. Define the event $D_0\subseteq \Pi_{\Z^2\setminus(a_1,a_1)}$ as,
\[D_0=\{\omega\in \Pi_{(\Z^2\setminus(a_1,a_1))}:\omega \cap \{\xi_{a_1,a_1}=0\} \in  D_2 \}.\]
Note that $D_0$ is independent of $\xi_{a_1,a_1}$. Also as $D_2$ is an increasing event in $\xi_{a_1,a_1}$, hence $D_0\cap \{\xi_{a_1,a_1}\geq x\}\subseteq D_2$ for all $x\geq 0$. Also, since $D_0\cap \{\xi_{a_1,a_1}=0\}=D_2\cap\{\xi_{a_1,a_1}=0\}$, hence,
\[\P(D_0)= \P(D_2|\xi_{a_1,a_1}=0)\geq 1-e^{-cC^{1/2}}.\]

Since $D_1$ is also a high probability event, one can choose $C$ a large constant (not depending on $n$) such that $\P(D_0\cap D_1)\geq \frac{1}{2}$. Observe that $D_1\cap D_0\cap \{\xi_{(a_1,a_1)}>2C^2\}\subseteq D_1\cap D_2\cap \{\xi_{(a_1,a_1)}>2C^2\}$. Also it follows from the proof of Lemma \ref{l:diagpos}, that on $D_1\cap D_2\cap \{\xi_{(a_1,a_1)}>2C^2\}$, $\Gamma_1$ touches the point $(a_1,a_1)$. As $D_1\cap D_0 \in \mathcal{F}_n$,
\[(T_1(C_1)-\E(T_1(C_1)|\mathcal{F}_n))^2\geq (T_1(C_1)-\E(T_1(C_1)|\mathcal{F}_n))^2\ind_{D_1\cap D_0}=(T_1(C_1)\ind_{D_1\cap D_0}-\E(T_1(C_1)\ind_{D_1\cap D_0}|\mathcal{F}_n))^2.\]
Also, 
\begin{eqnarray*}
&&T_1(C_1)\ind_{D_1\cap D_0}-\E(T_1(C_1)\ind_{D_1\cap D_0}|\mathcal{F}_n)\\
&=& T_1(C_1)\ind_{D_1\cap D_0\cap \{\xi_{(a_1,a_1)}>2C^2\}}-\E(T_1(C_1)\ind_{D_1\cap D_0\cap \{\xi_{(a_1,a_1)}>2C^2\}}|\mathcal{F}_n)\\
&&+(T_1(C_1)\ind_{D_1\cap D_0\cap \{\xi_{(a_1,a_1)}\leq 2C^2\}}-\E(T_1(C_1)\ind_{D_1\cap D_0\cap \{\xi_{(a_1,a_1)}\leq 2C^2\}}|\mathcal{F}_n)).
\end{eqnarray*}

Clearly the second summand is at most $2C^2$ in absolute value (by a similar calculation as done for the upper bound of $\tau$). Hence enough to show that the first summand is bigger than $4C^2$ with a positive probability that does not depend on $n$.

On $D_1\cap D_0\cap \{\xi_{(a_1,a_1)}>2C^2\}$, $\Gamma_1$ passes through $(a_1,a_1)$, hence,
\[ T_1(C_1)\ind_{D_1\cap D_0\cap \{\xi_{(a_1,a_1)}>2C^2\}}=(X_1+\xi_{(a_1,a_1)}+X_2)\ind_{D_1\cap D_0\cap \{\xi_{(a_1,a_1)}>2C^2\}},\]
and,
\[\E(T_1(C_1)\ind_{D_1\cap D_0\cap \{\xi_{(a_1,a_1)}>2C^2\}}|\mathcal{F}_n)=((X_1+X_2)\E(\ind_{\xi_{(a_1,a_1)}>2C^2})+\E(\xi_{(a_1,a_1)}\ind_{\xi_{(a_1,a_1)}>2C^2}))\ind_{D_1\cap D_0}\]

Let $D_3:=\{\xi_{(a_1,a_1)}\ind_{\{\xi_{(a_1,a_1)}>2C^2\}}\geq \E(\xi_{(a_1,a_1)}\ind_{\{\xi_{(a_1,a_1)}>2C^2\}})+4C^2\}$.
Note that $\{\ind_{\{\xi_{(a_1,a_1)}>2C^2\}}\geq \E(\ind_{\{\xi_{(a_1,a_1)}>2C^2\}})\}=\{\xi_{(a_1,a_1)}>2C^2\}\supset D_3$. Let $\P(D_3)=p$. 

As $\P(D_0\cap D_1)\geq \frac{1}{2}$ and $D_0\cap D_1$ is independent of $D_3$, hence $\P(D_0\cap D_1 \cap D_3)\geq \frac{p}{2}$. Hence with probability atleast $\frac{p}{2}$, $T_1(C_1)\ind_{D_1\cap D_0}-\E(T_1(C_1)\ind_{D_1\cap D_0}|\mathcal{F}_n)\geq 2C^2$. As $C,p$ are constants not depending on $n$, this proves the claim.
\end{proof}


\begin{lemma}\label{l:unifcont}
Fix $b\in \N, I=[-b,b]$ and $A\subseteq \{0,1\}^{I}$. For $s\geq 0$, define $\rho(s):=\P(\eta_s(I)\in A)$. Then $\rho(s)$ is uniformly continuous in $s$.
\end{lemma}
\begin{proof}
To see this, note that, for any $\delta>0$,
\[\rho(s+\delta)=\P(\eta_{s+\delta}(I)\in A|\xi_{(0,0)}>\delta)\P(\xi_{(0,0)}>\delta)+\P(\eta_{s+\delta}(I)\in A|\xi_{(0,0)}\leq\delta)\P(\xi_{(0,0)}\leq\delta).\]
If $\xi_{(0,0)}>\delta$, then the exponential clock at site $0$ of the TASEP has not yet ticked, and as the TASEP starts from step initial conditions, so $\P(\eta_{s+\delta}(I)\in A|\xi_{(0,0)}>\delta)=\P(\eta_s(I)\in A)$. Also as $\P(\xi_{(0,0)}>\delta)=e^{-r\delta}$ Hence,
\[|\rho(s+\delta)-\rho(s)|\leq 2(1-e^{-\delta}),\]
which shows the uniform continuity of $\rho(s)$.
\end{proof}

\subsection{Regularity estimate and uniform integrability used in Section \ref{s:couple}}

The scaling exponent $2/3$ in the transversal fluctuation of the geodesic was identified in \cite{J00} and \cite{BSS14}. However, for our purpose, a more refined and local fluctuation estimate shall be useful, which we quote from \cite{BSS17++}.

\begin{theorem}[\cite{BSS17++}, Corollary $2.4(a)$]\label{t:carsestimate}
Let $\psi>1$ and $m\in [\frac{1}{\psi},\psi]$ be fixed. Let $\Gamma$ be the geodesic from $(0,0)$ to $(n,mn)$ and $\mathcal{S}$ be the line segment joining $(0,0)$ to $(n,mn)$. For $\l\in \Z$, let $\mathcal{S}(\l)$ be such that $(\l,\mathcal{S}(\l))\in \mathcal{S}$, $\Gamma(\ell)\in \Z$ be the maximum number such that $(\ell,\Gamma(\ell))\in \Gamma$ and $\Gamma^{-1}(\ell)\in \Z$ be the maximum number such that $(\Gamma^{-1}(\ell),\ell)\in \Gamma$. 
Then there exist positive constants $n_0, \l_0, s_0, c$ depending only on $\psi$, such that for all $n\geq n_0, s\geq s_0, \l\geq \l_0$,
\[P[|\Gamma(\l)-\mathcal{S}(\l)|\geq s\l^{2/3}] \leq e^{-cs}.\]
\end{theorem}

We end this subsection with the following lemma that is used to get the uniform integrability conditions of random variables used earlier. This is a direct consequence of Theorem \ref{t:carsestimate} (or Theorem $2$ of \cite{BSS17++}).
\begin{lemma}\label{uniformintegrable} Let $k,R\in \N$ and $T_m^0$ denotes the length of the geodesic from $(0,0)$ to $(m,m)$ in the Exponential DLPP. Then 
\[\sup_{k,R} \E\left(\frac{T^{0}_{k+R}-T^{0}_k}{R}\right)^2<C^*<\infty.\]
\end{lemma}
\begin{proof}Let 
\[X=\min\{x|(x,k)\in \Gamma_{k+R} \mbox{\ \ or\ \ } (k,x)\in \Gamma_{k+R}\}.\]
If $X=k$, then $\Gamma_{k+R}$ passes through $(k,k)$, hence, $T_{k+R}^0-T_{k}^0=T^0_{(k,k),(k+R),k+R)}\overset{d}{=}T^0_R$. Similarly, if $X=k-\l$, then $T^0_{k+R}-T^0_{k}\prec T^0_{R+\l}$. We need to bound $\P(\frac{T^0_{k+R}-T^0_k}{R}>\sqrt{m})$ for each $m$ by a term not depending on $k,R$, such that the bound is summable. Hence for $\l$ small (i.e., $X$ large), we bound the probability by $\P(\frac{T^0_{R+\l}}{R}>\sqrt{m})$, and for $\l$ large, we bound the probability by the above Theorem \ref{t:carsestimate}. Let $\Gamma'$ be the geodesic $\Gamma_{R+k}$ viewed from $(R+k,R+k)$ to $(0,0)$, i.e.,  $\Gamma'(\l)=\Gamma_{k+R}(k+R-\l)$. We shall apply Theorem \ref{t:carsestimate} to the geodesic $\Gamma'$. Then, for $k\geq R$, we get for any $m\geq 1$,
\begin{eqnarray*}
&&\P\left(\frac{T^0_{k+R}-T^0_k}{R}>\sqrt{m}\right)\\
&\leq & \P\left(\frac{T^0_{k+R}-T^0_k}{R}>\sqrt{m}, X\geq (k-(R+m^{1/3}))_+\right)+\sum_{\l=R+m^{1/3}}^\infty \P\left(\frac{T^0_{k+R}-T^0_k}{R}>\sqrt{m}, k-X=\l\right)\\
&\leq & \P\left(\frac{T^0_{2R+m^{1/3}}}{R}\geq \sqrt{m}\right)+\sum_{\l=R+m^{1/3}}^\infty \P\left(\frac{T^0_{k+R}-T^0_k}{R}>\sqrt{m}, k-X=\l\right)\\
&\leq & \P\left(\frac{T^0_{2R+m^{1/3}}}{2R+m^{1/3}}\geq \frac{m^{1/6}}{2}\right)+\sum_{\l=R+m^{1/3}}^\infty \P\left(|\Gamma'(R+\l)-(R+\l)|\geq \l\right)\\
&\leq & 2^{2R+m^{1/3}}Ce^{-\frac{1}{4}(2R+m^{1/3})m^{1/6}}+
\sum_{s=m^{1/3}}^\infty \P\left(|\Gamma'_s-s|\geq \frac{s}{2}\right)\\
&\leq & Ce^{-cm^{1/3}}+Ce^{-cm^{1/10}}\leq Ce^{-cm^{1/10}}.
\end{eqnarray*}
The result follows immediately.
\end{proof}

\section{Appendix B: A central limit theorem for the slow bond model}
\label{s:clt}
As remarked before, here we provide a proof of a central limit theorem for the last passage time in the slow bond model. Recall that this is a consequence of the path getting pinned to the diagonal at a constant rate; and using Theorem \ref{meetondiagonal} one can argue that $T_n$ can be approximated by partial sums of stationary processes. This argument was outlined in \cite{BSS14}; we provide a complete proof here for the sake of completeness. Also observe that we do not really need this central limit theorem for the other results in this paper, however we believe it is an interesting result in its own right, hence the proof.

\begin{theorem}
\label{l:clt}
For any $r<1$; we have 
$$\frac{T_n^{(r)}-\E T_n^{(r)}}{\sqrt{{\rm Var}~T_n^{(r)}}}\Rightarrow N(0,1).$$
Furthermore, there exists $\sigma=\sigma(r)\in (0,\infty)$ such that $\lim_{n\to \infty} \frac{{\rm Var}~T_n^{(r)}}{n}=\sigma^2$.
\end{theorem}

\begin{proof} We suppress the dependence on $r$, and write $T_n^{(r)}$ simply as $T_n$.
Note that $\frac{T_n-\E(T_n)}{\sqrt{n}}=\frac{T_n-T_{n^{1/3}}-\E(T_n-T_{n^{1/3}})}{\sqrt{n}}+R$, where $R:=\frac{T_{n^{1/3}}-\E(T_{n^{1/3}})}{\sqrt{n}}\overset{\P}{\rightarrow}0$. Also $T_n-T_{n^{1/3}}=\sum_{i=n^{1/3}}^{n-1}\left(T_{i+1}-T_i\right)$. Define
 \[X_i:=T_{n^{1/3}+i}-T_{n^{1/3}+i-1} \mbox{ for } i=1,2,\ldots,t,\]
 where $t=n-n^{1/3}$. Then enough to show $\frac{\sum_{i=1}^{t}(X_i-\E(X_i))}{\sqrt{t}}\Rightarrow N(0,\sigma^2)$. As stated earlier, we would apply central limit theorem for stationary processes.

To this end, we first show that $X_1,X_2,\ldots, X_t$ is equal to a stationary sequence with high probability. Fix $1\leq k\leq t$. Then fix $\l\geq 0$ such that $k+\l\leq t$, let $\Gamma_\l$ be the geodesics from $(0,0)$ to $(n^{1/3}+k+\l,n^{1/3}+k+\l)$, and $\Gamma^k_\l$ be the geodesics from $(k,k)$ to  $(n^{1/3}+k+\l,n^{1/3}+k+\l)$. Let $E_k$ denote the event that there exists some $u\in \llbracket k,n^{1/3}+k\rrbracket$ such that $(u,u)\in \bigcap_{\l=0}^{t-k} (\Gamma_\l\cap \Gamma^k_\l)$. That is, $E_k$ denotes the event that all these paths meet together on the diagonal. Then by Corollary \ref{c:meetdiagpar}, $\P(E_k)\geq 1-e^{-cn^{1/12}}$. Let $Y_i^k:= T_{(k,k),(n^{1/3}+k+i,n^{1/3}+k+i)}-T_{(k,k),(n^{1/3}+k+i-1,n^{1/3}+k+i-1)}$. Clearly $(Y_{1}^k,Y_{2}^k,\ldots,Y_{t-k}^k)\overset{d}{=}(X_1,X_2,\ldots,X_{t-k})$. Note that on $E_k$, the differences in the lengths of geodesics starting from $(0,0)$ coincide with those starting from $(k,k)$. Hence, on $E_k$,

\[(X_{k+1},X_{k+2},\ldots,X_{t})=(Y_{1}^k,Y_{2}^k,\ldots,Y_{t-k}^k).\]
Let $E=\bigcap_{k=1}^{t} E_k$. Then $\P(E)\geq 1-ne^{-cn^{1/12}}$. And, for all $1\leq k\leq t$, there exist random variables $Y_{1}^k,Y_{2}^k,\ldots,Y_{t-k}^k$ such that $(Y_{1}^k,Y_{2}^k,\ldots,Y_{t-k}^k)\overset{d}{=}(X_1,X_2,\ldots,X_{t-k})$; and on $E$, $(X_{k+1},X_{k+2},\ldots,X_{t})=(Y_{1}^k,Y_{2}^k,\ldots,Y_{t-k}^k)$. 

Next we show that the sequence is $\alpha$-mixing. For this, we consider two sets $(X_1,X_2,\ldots,X_\l)$ and $(X_{\l+m+1},X_{\l+m+2},\ldots)$ such that the indices are separated by a distance of $m$. For any $\l,s\geq 1$ and $m\geq n^{1/3}$ such that $n^{1/3}+\l+m+s\leq n$, let $F$ denote the event that all the geodesics from $(0,0)$ to $(n^{1/3}+\l+m+s,n^{1/3}+\l+m+s)$, and all geodesics from $(n^{1/3}+\l+1,n^{1/3}+\l+1)$ to $(n^{1/3}+\l+m+s,n^{1/3}+\l+m+s)$ meet the diagonal simultaneously in the interval $\llbracket n^{1/3}+\l+1,n^{1/3}+\l+m \rrbracket$. Then as in previous paragraph, using Corollary \ref{c:meetdiagpar} and union bound (and the fact that $m\geq n^{1/3}$), it follows that $\P(F)\geq 1-e^{-cm^{1/4}}$ for some absolute positive constant $c$. For $j\geq \l+2$, define $Z_s=T_{(n^{1/3}+\l+1,n^{1/3}+\l+1),(n^{1/3}+j,n^{1/3}+j)}-T_{(n^{1/3}+\l+1,n^{1/3}+\l+1),(n^{1/3}+j-1,n^{1/3}+j-1)}$ to be the difference in the lengths of the corresponding geodesics starting from $(n^{1/3}+\l+1,n^{1/3}+\l+1)$ instead of $(0,0)$. Then, as before, on $F$, 
\[(X_{\l+m},X_{\l+m+1},\ldots)=(Z_{\l+m},Z_{\l+m+1},\ldots).\] 
Now for $A=f(X_1,X_2,\ldots,X_{\l})$,  $B=g(X_{\l+m+1},X_{\l+m+2},\ldots)$, and $B':=g(Z_{\l+m+1},Z_{\l+m+2},\ldots)$,
\begin{eqnarray*}
&&\left|\P(A\cap B)-\P(A)\P(B)\right|\\
&\leq&|\P(A\cap B\cap F)-\P(A)\P(B\cap F)|+2\P(F^c)\\
&=&|\P(A\cap B'\cap F)-\P(A)\P(B'\cap F)|+2\P(F^c)\\
&\leq &|\P(A\cap B')-\P(A)\P(B')|+4\P(F^c)\\
&=& 4\P(F^c)\leq 4e^{-cm^{1/4}},
\end{eqnarray*}
where we have used the fact that $A$ and $B'$ are independent. 

It is easy to see using Proposition \ref{deviationfromdiagonal} and Theorem \ref{meetondiagonal} that the geodesics $\Gamma_n$ and $\Gamma_{n-1}$ meet the diagonal simultaneously in the interval $\llbracket n-h,n \rrbracket$ with probability atleast $1-e^{-ch^{1/4}}$. From this it is not too hard to see that $\sup_n\E(T_n-T_{n-1})^{12}<\infty$. Hence, following the proof of Central Limit Theorem for stationary processes, (see e.g. Theorem 27.4 in \cite{billingsley}), with obvious modifications, the theorem follows. That $\sigma>0$ follows from the following Proposition \ref{p:lb}. This completes the proof.
\end{proof}

The following proposition shows that $\sigma>0$ in Theorem \ref{l:clt}.
\begin{proposition}
\label{p:lb}
Let $T_{n}$ denote the last passage time from $(0,0)$ to $(n,n)$ in the slow bond model. There exists $C>0$ such that ${\rm Var}~T_n\geq Cn$ for all $n$.
\end{proposition}

Recall that the individual passage time of vertex $v$ is denoted by $\xi_{v}$. We shall decompose ${\rm Var} (T_n)$ by revealing vertex weights in $[0,n]^2$ in some order. First fix a bijection $\pi: [n^2]\rightarrow [0,n]\times [0,n]$. Let $\cf_{i}$ denote the $\sigma$-field generated by $\{ \xi_{\pi(1)}, \xi_{\pi(2)}, \ldots , \xi_{\pi(i)}\}$. Considering the Doob martingale $M_i:=\E[T \mid \cf_i]$, it follows that we have 

\begin{equation}
\label{e:decomposev}
{\rm Var}(T)= \E \left[ \sum_{i=1}^{n^2} {\rm Var} \left(M_i\middle | \cf_{i-1}\right) \right]
\end{equation}

Also let $D\subseteq [n^2]$ denote the set such that $\pi(D)$ is the set of all vertices on the diagonal. Clearly 

\begin{equation}
\label{e:decomposev2}
{\rm Var}(T)\geq  \E \left[ \sum_{i\in D} {\rm Var} \left(M_i\middle | \cf_{i-1}\right) \right].
\end{equation}

The proposition will follow from the next lemma which provides a lower bound on the individual terms in the above sum. 

\begin{lemma}
\label{l:lb}
Let $M_i$ be as above and let $\gamma$ denote the geodesic from $(0,0)$ to $(n,n)$. Then for each $i\in D$ 
$${\rm Var} \left(M_i\middle | \cf_{i-1}\right) \geq h\biggl(\P [\pi(i)\in \gamma \mid \cf_{i-1}]\biggr)$$
where $h: \R_{+}\rightarrow \R_{+}$ is a function such that  $h(t)$ is bounded away from $0$ (uniformly in $n$) as soon as $t$ is bounded away from $0$.
\end{lemma} 

We postpone the proof of this lemma for the moment and first show how this implies Proposition \ref{p:lb}.

\begin{proof}[Proof of Proposition \ref{p:lb}]
First observe that since the expected number of vertices on the diagonal that $\gamma$ intersects in linear it follows by Cauchy-Schwarz inequality that 
\begin{equation}
\label{e:e1}
\E\biggl[\sum_{i\in D}\P[\pi(i)\in \gamma \mid \cf_{i-1}]^2\biggr] \geq \delta_1 n
\end{equation}
for all $n$ sufficiently large for some $\delta_1>0$. Notice further that for any $\delta_2>0$, we have 
\begin{eqnarray*}
\E\biggl[\sum_{i\in D}\P[\pi(i)\in \gamma \mid \cf_{i-1}]^2\biggr] &\leq & \delta_2\E\biggl[\sum_{i\in D}\P[\pi(i)\in \gamma \mid \cf_{i-1}]\biggr] \\
&+& \E\biggl[\sum_{i\in D}\P[\pi(i)\in \gamma \mid \cf_{i-1}]^2I(\P[\pi(i)\in \gamma \mid \cf_{i-1}]> \delta_2)\biggr]\\
& \leq & 2\delta_2 n + \E\biggl[\sum_{i\in D}I(\P[\pi(i)\in \gamma \mid \cf_{i-1}]> \delta_2)\biggr].
\end{eqnarray*}
By choosing $\delta_2$ sufficiently small compared to $\delta_1$ we get 
$$\E\biggl[\sum_{i\in D}I(\P[\pi(i)\in \gamma \mid \cf_{i-1}]> \delta_2)\biggr] \geq \frac{\delta_1 n}{2}$$
which implies the desired linear lower bound on $V_n$ using Lemma \ref{l:lb}.
\end{proof}

\begin{proof}[Proof of Lemma \ref{l:lb}]
Fix $i\in D$ and condition on $\cf_{i-1}$. Let $T^*$ be the last passage time in the environment where $\xi_{\pi(i)}$ is resampled by an independent copy $Z_{\pi(i)}$. Observe that 
$$M_i-M_{i-1}= \E[T-T^*\mid \cf_{i}].$$ Let the optimizing paths in the two environments be denote by $\gamma_1$ and $\gamma_2$ respectively. Let $x_0>0$ and set $w_0:=W(x_0)$ where $W(x):= \P[\pi(i)\in \gamma \mid \xi_{\pi(i)}=x, \cf_{i-1}]$
is an increasing function of $x$. Define the events
$$A_1:=\{\pi(i)\in \gamma_1~\text{if}~\xi_{\pi(i)}\geq x_0\};\quad A_2:=\{\pi(i)\in \gamma_2~\text{if}~Z_{\pi(i)}\geq x_0\}.$$
Clearly, $A_1=A_2:=A$ as the environments differ only in the weight of vertex $\pi(i)$  and also notice that $A$ is independent of $\xi_{\pi(i)}, Z_{\pi(i)}$. Further observe that $\P[A\mid \cf_{i-1}]\geq w_0$. Indeed, $A\cap \{\xi_{\pi(i)}\geq x_0\}=\{\pi(i)\in \gamma_1, \xi_{\pi(i)}\geq x_0\}$ and hence  $\P[A\cap \{\xi_{\pi(i)}\geq x_0\}] \geq w_0\P[\xi_{\pi(i)}\geq x_0]$. The desired inequality follows from observing that $A$ and $\{\xi_{\pi(i)}\geq x_0\}$ are conditionally independent given $\cf_{i-1}$. Observe that on $\{\xi_{\pi(i)}> (\ell+2)x_0\}$, we have 
$$T-T^* \geq \ell x_01_{\{Z_{\pi(i)}\in [x_0,2x_0], A\}}+ ((\ell +2)x_0-Z_{\pi(i)})1_{\{Z_{\pi(i)}>(\ell +2) x_0\}}$$
and hence 
$$\E[T-T^*\mid \cf_{i}] \geq \ell x_0w_0\P\left[Z_{\pi(i)}\in [x_0,2x_0]\right]- q(x_0,\ell)$$
where $q(x_0,\ell):=\E (Z_{\pi(i)}-(\ell +2)x_0)1_{\{Z_{\pi(i)}>(\ell +2) x_0\}}$ decreases to $0$ as $\ell$ increases, hence by choosing $\ell=\ell(x_0)$ sufficiently large , on $\{\xi_{\pi(i)}> (\ell+2)x_0\}$, we have  
$$M_i-M_{i-1} \geq x_0w_0.$$
It follows that 
$${\rm Var} \left(M_i\middle | \cf_{i-1}\right)= \E[(M_i-M_{i-1})^2\mid \cf_{i-1}]\geq x_0^2w_0^2 \P[\xi_{\pi(i)}\geq (\ell+2)x_0].$$
The proof of the lemma is completed by observing that exponential distribution has unbounded support and hence if $\mathbf{p}:=\P[\pi(i)\in \gamma \mid \cf_{i-1}]$ is bounded away from $0$, then one can choose $x_0=x_0(\mathbf{p})$ and $w_0=w_0(\mathbf{p})$ to be bounded away from $0$ as well.  
\end{proof} 

\end{document}